\documentclass{amsart}
\newtheorem{theorem}{Theorem}[section] % 1st argument is your name for it
\newtheorem{lemma}[theorem]{Lemma}     % 2nd argument is what is printed

\newtheorem{remark}{Remark}
\usepackage{dsfont, amsmath, amsfonts, amssymb,accents}
\usepackage{graphicx,color}
\usepackage{srcltx}

\numberwithin{equation}{section}

\newcommand{\eqskip}{ \vspace*{2mm}\\ }
\def\tht{\theta}
\def\Om{\Omega}
\def\om{\omega}
\def\e{\varepsilon}
\def\g{\gamma}

\def\l{\lambda}
\def\p{\partial}
\def\D{\Delta}
\def\bs{\backslash}

\def\a{\alpha}
\def\b{\beta}

\def\d{\delta}
\def\L{\Lambda}
\def\z{\zeta}

\def\vp{\varphi}

\def\Odr{\mathcal{O}}
\def\H{W_2}

\def\jp{{\kappa'}}

\def\di{\,d}

\def\RR{\mathds{R}}
\def\CC{\mathds{C}}

\def\PS{\mathcal{I}}

\def\He{\mathcal{H}_\e}
\def\he{\mathfrak{h}_\e}
\def\Ho{\mathcal{H}_0}
\def\ho{\mathfrak{h}_0}

\def\Hd{-\D_\om^{(D)}}
\def\Hn{-\D_\om^{(N)}}

\def\bs#1{\boldsymbol{#1}}
\newcommand{\ds}{\displaystyle}

 \DeclareMathOperator{\RE}{Re}
\DeclareMathOperator{\IM}{Im} 
\DeclareMathOperator{\discspec}{\sigma_{d}}

\DeclareMathOperator{\dist}{dist} 

\DeclareMathOperator{\Dom}{\mathcal{D}}
\DeclareMathOperator{\Tr}{Tr} 
\DeclareMathOperator{\Det}{det} \DeclareMathOperator{\Div}{div}
\DeclareMathOperator{\hs}{\mathcal S}

\newcounter{assumption}
\allowdisplaybreaks

\title[Deformations of compact flat hypersurfaces]{On the spectrum of deformations of
compact double-sided flat hypersurfaces}

\date{}
\author{Denis Borisov \and Pedro Freitas}
\address{
Department of Physics and Mathematics, Bashkir State Pedagogical
University, October rev. st., 3a, 450000, Ufa, Russia
}\email{borisovdi@yandex.ru}
\address{Department of Mathematics,
Faculdade de Motricidade Humana (TU Lisbon) {\rm and} Group of
Mathematical Physics of the University of Lisbon\\ Complexo
Interdisciplinar, Av.~Prof.~Gama Pinto~2\\ P-1649-003 Lisboa,
Portugal}\email{freitas@cii.fc.ul.pt}


\begin{document}
\maketitle

\begin{abstract}
We study the asymptotic behaviour of the eigenvalues of the
Laplace--Beltrami operator on a compact hypersurface in $\RR^{n+1}$
as it is {\it flattened} into a singular double--sided flat hypersurface.
We show that the limit spectral problem corresponds to the Dirichlet and Neumann
problems on one side of this flat (Euclidean) limit, and derive an explicit
three-term asymptotic expansion for the eigenvalues where the remaining two terms
are of orders $\e^2\log\e$ and $\e^2$.
\end{abstract}

\section{Introduction}
In recent years there have been several papers studying the effect that flattening a domain
has on the eigenvalues of the Laplace operator~\cite{bc,bofr1,bofr2,frso}; see also the
books~\cite{Na,Pa} and the references therein for similar problems with boundary
conditions other than Dirichlet. In these papers the main objective
has been the derivation of the asymptotics of these eigenvalues in terms
of a scalar parameter measuring how thin the domain becomes in one direction, as this
parameter approaches zero. As far as we are aware, almost if not all such existing
examples in the literature are concerned with domains in Euclidean space where
the limiting problem degenerates to a domain of zero measure and therefore eigenvalues
approach infinity.

A slightly different set of problems which has been considered consists
of domains which are perturbations of singular sets such as thin tubular neighbourhoods of
graphs, i.e., domains which locally are like thin tubes -- see~\cite{EP1,EP3}, for instance,
and also~\cite{Gr2} for a review. As in the papers cited above, again the limiting domains
have zero measure and the spectrum behaves in quite a different way from
the model considered here.

In this paper we study a situation which, although different from that described in
the first paragraph, has in common with it the process by which the limiting domain is
approached. More precisely, consider the case of a given domain $\Omega$
in $\RR^{n+1}$ satisfying certain restrictions which for the purpose here may be stated
roughly as being bounded from above and below by the graphs of two functions -- see
Section~\ref{formulation} for a precise formulation. The domain $\Omega$ is then flattened
towards a domain $\omega$ in $\RR^{n}$ via a (continuous) one-parameter family of domains
$\Omega_{\e}$. These domains are
obtained as the functions mentioned above are
multiplied by the parameter $\e$.
The problem that shall concern us here is the study of the evolution of the eigenvalues of
the Laplace-Beltrami operator on the one-parameter family of compact hypersurfaces $\hs_{\e}$
which are the boundaries of the domains $\Omega_{\e}$ described above, as $\e$ approaches
zero. One of the differences in this instance is that while the domain $\Omega_{0}$ has zero
$(n+1)-$measure as stated above, $\hs_{0}$ retains positive $n-$measure, developing instead
a singularity on the boundary of the domain $\omega$ (when considered as a domain in $\RR^{n}$).
We thus expect these eigenvalues to remain finite as the parameter $\e$ approaches zero,
and to converge to a limiting spectral problem on the double--sided flat hypersurface. This
is indeed the case, and the relevant spectral problems turn out to be the Dirichlet
and Neumann problems on the domain $\omega$, with the two next asymptotic terms after that
being of orders $\e^{2}\log\e$ and $\e^{2}$. These results have been announced in~\cite{bofr3}.

\begin{figure}
\begin{center}
\includegraphics[scale=0.6]{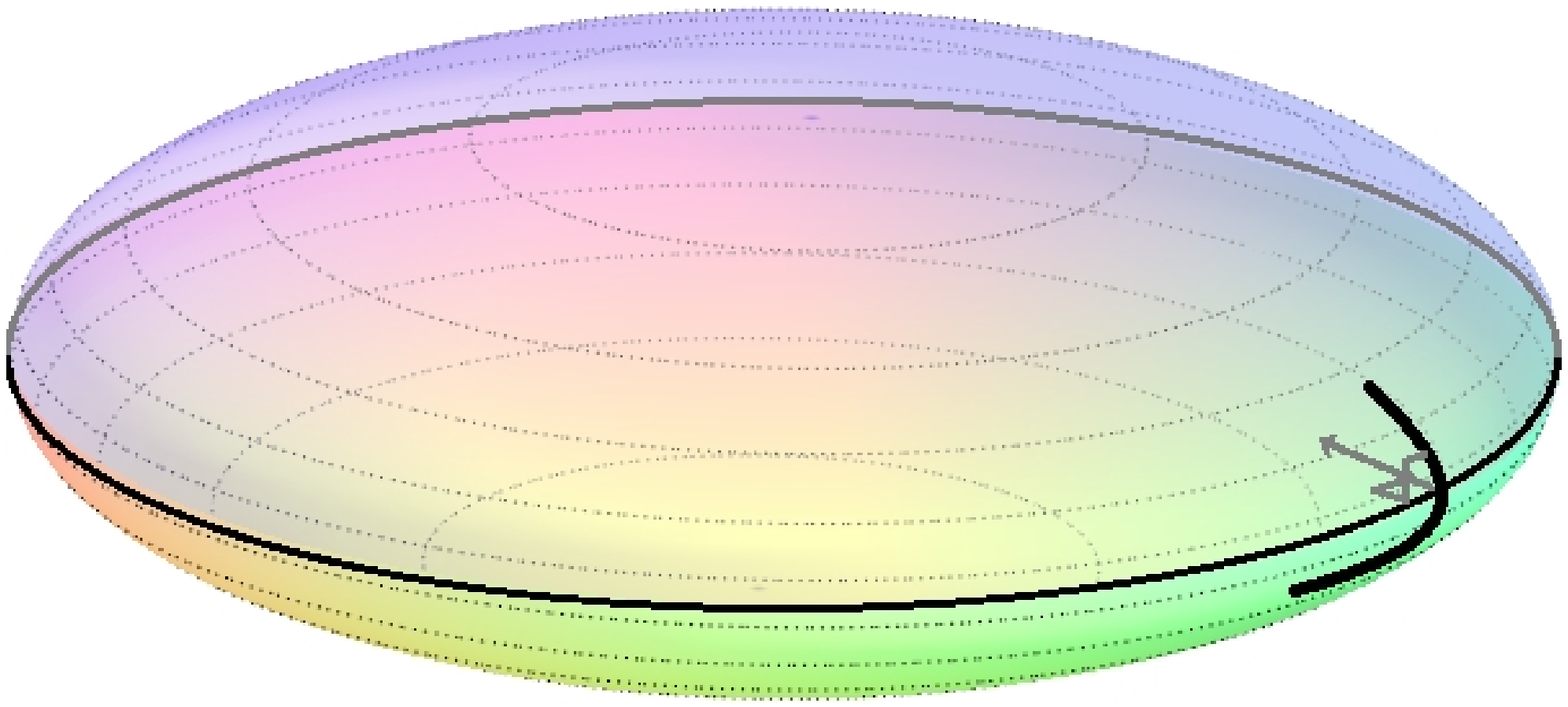}
\caption{Surface $\hs_\e$ with a cross-section at the edge}\label{crosssect}
\end{center}
\end{figure}

In order to understand the origin of the $\e^2\log\e$ term in the expansion, it turns
out that it is sufficient to consider the case where $n$ equals one, that is when the boundary
is basically $S^{1}$. Because of this, it is not necessary to take into consideration the geometric
intrincacies of the problem which appear in higher dimensions and it is possible to obtain
the full description of eigenvalues in terms of elliptic integrals.

More precisely, for an ellipse of radii $1$ and $\e$ we have that the eigenvalues are
given by
\[
\lambda_{k}(\e) = \frac{\ds k^2 \pi^2}{\ds E^2(1-\e^2)},
\]
for $k\in\mathbb{Z}$ and where
\[
E(m) = {\ds\int}_{0}^{\pi/2} \sqrt{1-m \sin^2(\theta)}\;{\rm d}\theta
\]
is the complete elliptic integral of the second type yielding one quarter of
the perimeter of the ellipse for $m=1-\e^2$.

Combining the above with the asymptotic expansion for $E$ yields
\[
\lambda_{k}(\e) = \frac{\ds k^2 \pi^2}{4} + \frac{\ds k^2 \pi^2}{\ds 4}\e^2 \log\e+
\frac{\ds k^2 \pi^2}{\ds 2}\left(\frac{\ds 1}{\ds 4}-\log 2\right)\e^2+\Odr(\e^{2+\rho}),
\;\;\rho\in(0,1).
\]

In some sense, the purpose of the analysis that we shall carry out in what follows
is to show that the above result may actually be extended to higher dimensions.
It should be noted here that this expansion depends on the relation between the different variables
at the endpoints of the segment, which in this case is of the form $x_{1}^{2}+\e^2 x_{2}^2=1$.
Clearly different relations between the leading powers will lead to different expansions.

More generally, the issue is that the points of the boundary of $\Omega$ where there is a
tangent in the direction along which the domain is being flattened will play a special role.
Throughout the paper we assume this set of points to be contained in a hyperplane
orthogonal to the scaling direction, and that this tangency is simple. In the vicinity of these points we take the cross-section of our surface as indicated in Fig.~\ref{crosssect} which,
with the assumptions made, will be similar to the one-dimensional ellipse described above.
Our results then state that in the higher-dimensional case the asymptotics for the eigenvalues still
behave in a similar fashion and thus the logarithmic terms appearing above persist in this more
general setting.

Apart from the intrinsic interest of the behaviour of the spectrum close to double--sided
flat domains, we point out that such manifolds have appeared in the
literature in connection with eigenvalues as maximizers of the invariant eigenvalues among
all surfaces isometric to surfaces of revolution in $\RR^{3}$~\cite{abfr} and for hypersurfaces
of revolution diffeomorphic to a sphere and isometrically embedded in $\RR^{n+1}$~\cite{cde}.
In fact, it is shown in those papers that these optimal singular {\it double flat disks}
maximize the whole invariant spectrum and not just a specific eigenvalue.
Another source of interest for such asymptotic expansions lies with the fact that, in some
cases, they turn out to be fairly good approximations for low eigenvalues also for values of the
parameter $\e$ away from zero -- see~\cite{bofr1,bofr2,frei}.

We remark in passing that another problem for which it is conjectured that the optimal shape
is given by a double--sided flat disk is Alexandrov's conjecture relating the area and diameter
of surfaces of non--negative curvature.

The structure of the paper is as follows. In the next section we give a precise formulation of
the problem under consideration and state our main results, namely, the nature
of the limiting problem and the relation of the limit and approximating operators. This includes
the form of the asymptotic expansion and the expressions for the first three coefficients
and an application to the case of the surface of an ellipsoid.
Section~\ref{prelim} is then devoted to several preliminaries and auxiliary material used in
Sections~\ref{unifresconv} and~\ref{sec4}, where the proofs of the main
results are presented.
%
% The
%proofs of these results are then given in Sections.

\section{Problem formulation and main results\label{formulation}}

Let $x'=(x_1,\ldots,x_n)$, $x=(x',x_{n+1})$ be Cartesian coordinates
in $\RR^n$ and $\RR^{n+1}$ respectively, $n\geqslant 2$, $\om$ be a bounded
domain in $\RR^{n}$ with infinitely smooth boundary. Let also
$h_\pm=h_\pm(x')\in C^\infty(\om)\cap C(\overline{\om})$ denote
two arbitrary functions and define the manifold
\begin{equation}\label{1.0}
\hs_\e:=\{x: x'\in\overline{\om}, \quad x_{n+1}=\e h_+(x')\}\cup \{x:
x'\in\overline{\om}, \quad x_{n+1}=-\e h_-(x')\},
\end{equation}
where $\e$ is a small positive parameter. We assume $\hs_\e$ to be
infinitely differentiable and to have no self-intersections. To ensure this,
we make the following assumptions on $h_\pm$, the first of which
ensures the absence of self-intersections,
\begin{enumerate}\def\theenumi{A\arabic{assumption}}
\addtocounter{assumption}{1}
\item\label{A1} The relations
\begin{equation*}
h_+(x')+h_-(x')>0, \quad x'\in\om, \qquad h_+(x')=h_-(x')=0,\quad
x'\in\p\om,
\end{equation*}
hold true.
\end{enumerate}

To state the second assumption we need to introduce some additional
notation. Let $\nu=\nu(P)$, $P\in\p\om$, be the inward normal to $\p\om$, and
denote by $\tau$ the distance to a point measured in the direction of
$\nu$. Consider equations
\begin{equation}\label{1.2}
t=h_+(P+\tau \nu(P)),\quad t>0,\qquad t=-h_-(P+\tau \nu(P)),\quad
t<0.
\end{equation}
Our second assumption concerns the solvability of these equations
with respect to $\tau$ and implies the smoothness of $\hs_\e$ in a neighbourhood of
$\p\om$:
\begin{enumerate}\def\theenumi{A\arabic{assumption}}
\addtocounter{assumption}{1}
\item\label{A2} There exists $t_0>0$ such that for all
$t\in[-t_0,t_0]$, $P\in\p\om$, equations (\ref{1.2}) have a
unique solution given by
\begin{equation*}%\l%abel{1.3}
\tau=a(t,P)\in C^\infty([-t_0,t_0]\times\p\om),
\end{equation*}
such that
\begin{equation}\label{1.4}
\frac{\p^2 a}{\p t^2}>0 \quad \text{for all $P\in\p\om$}.
\end{equation}
\end{enumerate}

We observe that assumptions $(A1)$ and $(A2)$ imply that
\begin{equation*}%\l%abel{1.1}
h_+(x')\geqslant 0,\quad h_-(x')\leqslant 0 \quad \text{in a small
neighbourhood of $\p\om$}.
\end{equation*}

The main object of our study is the Laplace-Beltrami operator $\He$
on $\hs_\e$. We introduce it rigorously as the self-adjoint operator
associated with a symmetric lower-semibounded  sesquilinear form
\begin{equation*}%\l%abel{1.5}
\he[u,v]:=(\nabla u,\nabla v)_{L_2(\hs_\e)}\quad\text{on}\quad
\H^1(\hs_\e).
\end{equation*}
We recall that on an arbitrary manifold with metric tensor $g$ this
may be written in local coordinates $y=(y_1,\ldots,y_n)$ as
\begin{equation*}
-\Det^{-\frac{1}{2}} g \sum\limits_{i,j=1}^{n} \frac{\p}{\p y_i} g^{ij}\Det^{\frac{1}{2}} g \frac{\p}{\p y_j},
\end{equation*}
where $g^{ij}$ are the entries of the inverse to the metric tensor.
If in our case we take $x'$ as local coordinates on $\hs_\e$, then on
each side $\hs_\e^\pm$ the operator $\He$ may be written in the form
\begin{equation}\label{2.3a}
\He=-(1+\e^2|\nabla_{x'} h_\pm|^2)^{-\frac{1}{2}}
\Div_{x'} (1+\e^2|\nabla_{x'} h_\pm|^2)^{\frac{1}{2}} (\mathrm{E}+\e^2 \mathrm{Q}_\pm)^{-1}\nabla_{x'},
\end{equation}
where $\mathrm{E}$ is the $n\times n$ identity matrix and $\mathrm{Q}_\pm$
is the matrix with entries $\frac{\p h_\pm}{\p x_i}\frac{\p h_\pm}{\p x_j}$.
On the boundary $\p\om$ the coefficients of such operator have singularities,
and this is why in a neighbourhood of $\p\om$ it is more convenient to employ
the coordinates $(\tau,s)$, where $s$ are some local coordinates on $\p\om$.
We do not give here the expression of the operator $\He$ in such coordinates,
as it requires the introduction of additional (cumbersome) notation.

The purpose of the present paper is to describe the asymptotic behavior
of the resolvent and the spectrum of $\mathcal{H}_\e$ as $\e\to+0$.
In this limit, the hypersurface $\hs_\e$ collapses to a flat two-sided
domain $\bs\om=(\om_+,\om_-)$, where $\om_\pm$ are two copies of
$\om$ understood as the {\it upper} and {\it lower} sides of $\bs{\om}$.
Because of this, it is natural to expect that the limiting operator for
$\mathcal{H}_\e$ as $\e\to+0$ is the Laplacian on $\bs{\om}$, i.e.,
that on $\om_\pm$ subject to certain boundary conditions. Indeed,
this is true, and it is our first main result. Namely, we introduce
the space $L_2(\bs{\om})$ as consisting of the vectors
$\bs{u}=(u_+,u_-)$, where the functions $u_\pm$ are defined on
$\om_\pm$ and $u_\pm\in L_2(\om_\pm)$. We can natural identify
$L_2(\bs{\om})$ with $L_2(\om)\oplus L_2(\om)$. In the same way
we introduce the Sobolev spaces $\H^j(\bs{\om})$ assuming that for
each $\bs{u}\in \H^j(\bs{\om})$ the functions
$u_\pm\in\H^j(\om_\pm)$ satisfy the boundary conditions
\begin{equation}\label{1.6}
\frac{\p^i u_+}{\p\tau^i}\Big|_{\p\om}=(-1)^i \frac{\p^i
u_-}{\p\tau^i}\Big|_{\p\om},\quad i=0,1,\ldots,j-1.
\end{equation}
The meaning of these boundary conditions is that the functions
$u_\pm$ should be ``glued smoothly'' while moving from $\om_+$ to
$\om_-$ via $\p\om=\p\om_\pm$. We observe that $\H^j(\bs{\omega})$ is
embedded into $\H^j(\om)\oplus \H^j(\om)$, but does not coincide. It
is also clear that for any $u\in\H^1(\om)$ the function
$\bs{u}:=(u,u)$ belongs to $\H^1(\bs{\om})$. Similarly, if
$u\in\H^2(\om)$, $u|_{\p\om}=0$,  respectively, $u\in\H^2(\om)$,
$\frac{\p u}{\p\tau}\big|_{\p\om}=0$, then
$\bs{u}=(u,-u)\in\H^2(\bs{\om})$, respectively,
$\bs{u}=(u,u)\in\H^2(\bs{\om})$.

Let $\Ho$ be the self-adjoint operator in $L_2(\bs{\om})$ associated
with the closed symmetric lower-semibounded sesquilinear form
\begin{equation*}%\l%abel{1.7}
\ho[\bs{u},\bs{v}]:=(\nabla \bs{u},\nabla \bs{v})_{L_2(\bs\om)}
\quad\text{on}\quad \H^1(\bs{\om}).
\end{equation*}
By $\Dom(\cdot)$ we denote the domain of an operator, the symbol
$\|\cdot\|_{X\to Y}$ indicates the norm of an operator acting from the
Hilbert space $X$ to a Hilbert space $Y$.

Given any vector $\bs{u}=(u_+,u_-)$ defined on $\bs{\om}$, by
$\PS_\e\bs{u}$ we denote the function on $\hs_\e$ being $u_+(x')$ on
$\{x: x'\in\overline{\om}, \quad x_{n+1}=\e h_+(x')\}$ and $u_-(x')$
on $\{x: x'\in\overline{\om}, \quad x_{n+1}=-\e h_-(x')\}$. And vice
versa, given any function $u$ defined on $\hs_\e$, by $\PS_\e^{-1} u$
we denote the vector $\bs{u}=(u_+,u_-)$, where
$u_\pm=u_\pm(x'):=u(x')$, $x'\in\om$, $x_{n+1}=\e h_\pm(x')$.

\begin{theorem}\label{th1.1}
For each $z\in \CC\setminus\RR$ %D%
there exists $C(z)>0$ such that
%D%
the estimate
\begin{equation}\label{1.8}
\|(\He-z)^{-1}-\PS_\e(\Ho-z)^{-1}
\PS_\e^{-1}\|_{L_2(\hs_\e)\to\H^1(\hs_\e)}\leqslant C(z)\e^{2/3}
\end{equation}
holds true.
\end{theorem}

\begin{remark}\label{rm1.1}
The statement of this theorem includes the fact that the operator
$\PS_\e(\Ho-z)^{-1} \PS_\e^{-1}$ is well-defined as a bounded one from $L_2(\hs_\e)$
into $\H^1(\hs_\e)$.
\end{remark}

In view of the embedding of $\H^1(\bs{\om})$ into
$\H^1(\om)\oplus\H^1(\om)$, and the compact embedding of the latter
into $L_2(\om)\oplus L_2(\om)=L_2(\bs{\om})$, the operator $\He$ has
a compact resolvent. Hence, it has a pure discrete spectrum
accumulating only at infinity. The same is true for the Dirichlet
and Neumann Laplacians $\Hd$ and $\Hn$ on $\om$. Recall that $\Hd$
is the Friedrichs extension in $L_2(\om)$ of $-\D$ from
$C_0^\infty(\Om)$, and $\Hn$ is the self-adjoint operator in
$L_2(\om)$ associated with the sesquilinear form $(\nabla u,\nabla
v)_{L_2(\Om)}$ on $\H^1(\om)$. In what follows $\discspec(\cdot)$
denotes the discrete spectrum of an operator.

Our next result follows from Theorem~\ref{th1.1} and \cite[Thms. V\!I\!I\!I.23, V\!I\!I\!I.24]{RS1}.

\begin{theorem}\label{th2.2}
The eigenvalues of $\He$ converge to those of $\Ho$ as $\e$ goes to zero.
In particular, if $\l\not\in\discspec(\Ho)$, then $\l\not\in\discspec(\He)$ for $\e$
small enough. For each $m$-multiple eigenvalue
$\l\in\discspec(\Ho)$ there exist exactly $m$ eigenvalues
(counting multiplicities) of $\He$ converging to $\l$ as
$\e\to+0$. Let $\mathcal{P}_0$ be the projector on the eigenspace
associated with $\l$, $\mathcal{P}_\e$ be the total projector
associated with the eigenvalues of $\He$ converging to $\l$. Then
the convergence
\begin{equation*}
\|\mathcal{P}_\e - \PS_\e
\mathcal{P}_0\PS_\e^{-1}\|_{L_2(\hs_\e)\to\H^1(\hs_\e)}\to0,\quad
\e\to+0,
\end{equation*}
holds true.
\end{theorem}

Let now $\l$ be an eigenvalue of $\Ho$ with multiplicity $m$ and $\bs{\psi}_i=(\psi_+^{(i)},\psi_-^{(i)})$ be %the
associated eigenfunctions orthonormalized in $L_2(\bs{\om})$. It will be shown in the next section in Lemma~\ref{lm1.2}
that the asymptotics
\begin{equation}\label{2.15}
\psi_\pm^{(i)}(x')=\Psi_i^{(0)}(P) \pm \Psi_i^{(1)}(P)\tau + \Odr(\tau^2),\quad P\in\p\om,\quad
\tau\to+0,
\end{equation}
hold true, where
\begin{align*}
&\Psi_i^{(0)}=\psi_+^{(i)}\big|_{\p\om}=\psi_-^{(i)}\big|_{\p\om}\in C^\infty(\p\om),
\quad
\Psi_i^{(1)}=\frac{\p\psi_+^{(i)}}{\p\tau}\Big|_{\p\om}=-\frac{\p\psi_-^{(i)}}{\p\tau}\big|_{\p\om}\in C^\infty(\p\om)
\end{align*}

By $-\D_{\p\om}$ we denote  the Laplace-Beltrami operator on $\p\om$, where the metric $\mathrm{G}_{\p\om}$ on $\p\om$ is induced by the Euclidean one in $\mathds{R}^n$. For any smooth functions $u,v$ on $\p\om$,
we shall denote the pointwise scalar product of its gradients by $\nabla u \cdot \nabla v$.

Let
\begin{equation}\label{3.12a}
\om^\d:=\om\setminus\{x': 0<\tau<\d\}.
\end{equation}
Employing the coefficients of the asymptotics (\ref{2.15}), we introduce two real symmetric matrices
$\Lambda^{(0)}$, $\Lambda^{(1)}$ with entries

\begin{equation}
\begin{aligned}
\Lambda^{(0)}_{ij}:=& \int\limits_{\p\om} \frac{1}{a_2} \big( \l \Psi_i^{(0)}\Psi_j^{(0)}-\nabla\Psi_i^{(0)}\cdot \nabla\Psi_j^{(0)} +\Psi_i^{(1)}\Psi_j^{(1)} \big)\di\om
\end{aligned}\label{2.11a}
\end{equation}
and
\begin{equation}
\begin{aligned}
\Lambda^{(1)}_{ij}:=&-\lim\limits_{\d\to+0} \Bigg[
\frac{1}{2} \int\limits_{\om^\d} |\nabla_{x'}h_+|^2\big(\l\psi_+^{(i)} \psi_+^{(j)} -(\nabla_{x'}\psi_+^{(i)}, \nabla_{x'} \psi_+^{(j)})_{\mathds{R}^d} \big)\di x'
\\
&\hphantom{\lim\limits_{\d\to+0} \Bigg[}+ \frac{1}{2} \int\limits_{\om^\d} |\nabla_{x'}h_-|^2\big(\l\psi_-^{(i)} \psi_-^{(j)} -(\nabla_{x'}\psi_-^{(i)}, \nabla_{x'} \psi_-^{(j)})_{\mathds{R}^d} \big)\di x'
\\
&\hphantom{\lim\limits_{\d\to+0} \Bigg[}+\int\limits_{\om^\d} (\nabla_{x'}h_+,\nabla_{x'}\psi_+^{(i)})_{\mathds{R}^d}
(\nabla_{x'}h_+,\nabla_{x'}\psi_+^{(j)})_{\mathds{R}^d}\di x'
\\
&\hphantom{\lim\limits_{\d\to+0} \Bigg[}+\int\limits_{\om^\d}
(\nabla_{x'}h_-,\nabla_{x'}\psi_-^{(i)})_{\mathds{R}^d}
(\nabla_{x'}h_-,\nabla_{x'}\psi_-^{(j)})_{\mathds{R}^d}
  \di x'
\\
&\hphantom{\lim\limits_{\d\to+0} \Bigg[}+ \ln \d \int\limits_{\p\om} \frac{1}{4a_2} \big( \Psi_i^{(1)} \Psi_j^{(1)}+\l\Psi_i^{(0)}\Psi_j^{(0)} -\nabla\Psi_i^{(0)}\cdot \nabla\Psi_j^{(0)}\big)\di s\Bigg]
\\
&-\int\limits_{\p\om}\frac{1+4\ln 2+\ln a_2}{4a_2}
\big(\Psi_i^{(1)}\Psi_j^{(1)}+\l\Psi_i^{(0)}  \Psi_j^{(0)}-\nabla\Psi_i^{(0)}\cdot \nabla\Psi_j^{(0)}
\big)\di s,
\end{aligned}\label{2.11b}
\end{equation}
where
\begin{equation*}
a_2(P):=\frac{1}{2}\frac{\p^2 a}{\p t^2}(0,P).
\end{equation*}
It will be shown in Sec.~4 that the matrix $\Lambda^{(1)}$ is well-defined. By the theorem on simultaneous diagonalization of two quadratic forms, in what follows the eigenfunctions $\bs{\psi}_i$ are supposed to be orthonormalized in $L_2(\bs\om)$ and the matrix
%\begin{equation*}
$\Lambda^{(0)}+ \frac{1}{\ln\e} \Lambda^{(1)}$
%\end{equation*}
to be diagonal. The eigenfunctions $\bs{\psi}_i$ chosen in this way depend on $\e$, but it is clear that the norms $\|\psi_\pm^{(i)}\|_{C^k(\overline{\om})}$ are bounded uniformly in $\e$ for all $k\geqslant 0$, $i=1,\ldots,m$.

\begin{theorem}\label{th2.3}
Let $\l$ be an $m$-multiple eigenvalue of $\Ho$ and  $\bs{\psi}_i$, $i=1,\ldots,m$, be the associated eigenfunctions of $\Ho$ chosen as described
above.
Then there exist exactly
$m$ eigenvalues $\l_k(\e)$, $k=1,\ldots,m$
(counting multiplicity) of $\He$ converging to $\l$. These
eigenvalues satisfy the asymptotic expansions
\begin{equation}\label{2.16}
\l_k(\e)=\l+\e^2\ln\e\, \mu_k\left(%D%
\frac{1}{\ln\e}%D%
\right)+\Odr(\e^{2+\rho}),
\end{equation}
where $\mu_k$ are the eigenvalues of the matrix $\Lambda^{(0)}+ \frac{1}{\ln\e} \Lambda^{(1)}$, and $\rho$ is any constant in $(0,1/2)$. The eigenvalues $\mu_k\left(%D%
\frac{1}{\ln\e}%D%
\right)$ are holomorphic in $\frac{1}{\ln\e}$ and converge to the eigenvalues of $\Lambda^{(0)}$ as $\e\to0$.
\end{theorem}

In addition to the asymptotic expansions for the eigenvalues $\l_i(\e)$ given in this theorem, we also obtain
the asymptotics for the total projector associated with these eigenvalues. However, to formulate this result
we have to introduce additional notation and it is thus more convenient to postpone its statement which will
them be made at the end of Sec.~\ref{sec4} -- see Theorem~\ref{th4.7}.

Let us describe briefly the main ideas employed in the proof of the main results. The proof of the
uniform resolvent convergence in Theorem~\ref{th1.1} is based on the analysis of the quadratic forms
associated with the perturbed and the limiting operators and on the accurate estimates of the functions
in certain weighted Sobolev spaces. The proof of the first theorem uses essentially the method of
matching asymptotic expansions \cite{Il} for formal construction of the asymptotics for the
eigenfunctions associated with $\l_k(\e)$. These asymptotics are constructed as a combination of
outer and inner expansions. The former depends on $x'$ and its coefficients have singularities at
$\p\om$. In the vicinity of $\p\om$ we introduce a special rescaled variable
$\xi:=a^{1/2}(x_{n+1}\e^{-1},P)\e^{-1}$ as $x_{n+1}>0$ and
$\xi:=-a^{1/2}(x_{n+1}\e^{-1},P)\e^{-1}$ as $x_{n+1}<0$. This variable then describes the slope of
$\hs_\e$ in the vicinity of $\e$ -- see also equations (\ref{4.54}) giving the parametrization of
$\hs_\e$ in the vicinity of $\p\om$. After rewriting the eigenvalue equation in the variables $(\xi,s)$,
where $s$ are local coordinates on $\p\om$, its leading term is in fact the Laplace-Beltrami operator on
the ellipse giving rise to the logarithmic terms in the asymptotics for both the eigenvalues and the
eigenfunctions.

Despite the fact that we are only presenting the leading terms of the asymptotics for $\l_k(\e)$
and for the associated total projector in Theorems~\ref{th2.3} and~\ref{th4.7}, respectively, our
approach also allows us to construct the complete asymptotic expansions if required. Although
this would need to be checked in a way similar to what was done here for the first few terms,
the ansatzes (\ref{4.1}) and (\ref{4.116}) suggest that the complete asymptotic expansion for
the eigenvalues should be
\begin{equation*}
\l_k(\e)=\l+\e^2\ln\e\mu_k(\e)+\sum\limits_{i=2}^{\infty} \e^{2i}\ln^i\e \mu_k^{(i)}\left(%D%
\frac{1}{\ln\e}%D%
\right),
\end{equation*}
where $\mu_k^{(i)}$ are functions holomorphic in $\frac{1}{\ln\e}$. These higher-order terms would
then still reflect the behaviour observed in the ellipse example given in the Introduction.

Although the above formulas for $\L_{ij}^{(0)}$ and (specially) $\L_{ij}^{(1)}$ may look quite
cumbersome at a first glance, they will actually simplify when computed for particular cases as some of
the terms involved will vanish depending on whether we are considering Dirichlet or Neumann boundary
conditions on $\partial\omega$. We note that a similar effect was already present when computing the
coefficients in the expansions obtained in~\cite{bofr1,bofr2}. This is particularly clear in
the second of these papers dealing with dimensions higher than two, where the general expression is
quite complicated and needs to be computed specifically in each case. When this is done for general
ellipsoids in any dimension, for instance, it yields a much simpler one-line expression.

We shall illustrate this by considering a thin ellipsoidal surface. To this end
take $\om$ to be the unit disk centred at the origin with
\begin{equation}\label{ex2}
h_\pm(x'):=\sqrt{1-r^2},\quad r=|x'|,\quad\tau=1-r,\quad a_2=\frac{1}{2}.
\end{equation}
In this instance the limiting eigenvalues may be found via separation of variables and they will be of the
form $\kappa^2$, where $\kappa$ are the zeroes of the Bessel function $J_\kappa$ and its derivative $J_\kappa'$,
corresponding to eigenfunctions satisfying Dirichlet and Neumann boundary conditions on $\p\om$, respectively.
The following examples illustrating both cases are taken from~\cite{bofr3}, where the details may be found.

We consider the case of Dirichlet boundary conditions first, i.e.,
\begin{align*}
&J_0(\kappa)=0,\quad \l=\kappa^2,\quad \psi(x)=-\frac{J_0(\kappa r)}{\sqrt{2\pi}J_1(\kappa)},
\\
&\bs{\psi}=(\psi,-\psi),
\quad
\Psi^{(0)}=0,\quad \Psi^{(1)}=-\frac{\kappa}{\sqrt{2\pi}}.
\end{align*}
Substituting these formulas and (\ref{ex2}) into (\ref{2.11a}) and (\ref{2.11b}), we then obtain
\begin{equation*}
\L^{(0)}_{11}=2\l
\end{equation*}
and
\begin{align*}
\L^{(1)}_{11}=& -\frac{\l}{J_1^2(\kappa)}\int\limits_0^{1} \frac{r^3}{1-r^2}\Big( J_0^2(\kappa r)+J_1^2(\kappa r)-J_1^2(\kappa)\Big)\di r-\l \ln 2.
\end{align*}
The asymptotics~(\ref{2.16}) thus become
\begin{equation*}
\l_\kappa(\e)=\l+\e^2\left(2\l\ln\e+\L^{(1)}_{11}\right)+\Odr(\e^{2+\rho})
\end{equation*}
and, for a particular eigenvalue, the remaining integral
may be computed numerically. We illustrate this by considering the case corresponding to the
first Dirichlet eigenvalue on the disk which yields
\begin{align*}
\l_1(\e) = & j_{0,1}^2+\e^2(2j_{0,1}^2\ln\e+\L^{(1)}_{11})+\Odr(\e^{2+\rho})
\\
\approx & 5.7831 + 11.5664\, \e^2\ln\e - 6.0871\,\e^2+\Odr(\e^{2+\rho}).
\end{align*}

As an example of limiting multiple eigenvalue we consider the first nontrivial Neumann eigenvalue of
the disk. In two dimensions this is a double eigenvalue with associated (normalized) eigenfunctions given by
\[
\psi_1(x)=\frac{J_1(\jp r)\cos\tht}{J_0(\jp)\sqrt{\pi(\jp^2-1)}},\quad
\psi_2(x)=\frac{J_1(\jp r)\sin\tht}{J_0(\jp)\sqrt{\pi(\jp^2-1)}},
\]
where $\tht$ is the polar angle corresponding to $x$ and $\jp$ is the first nontrivial zero of
$J_{1}'$.

The eigenfunctions in $L_2(\bs{\om})$ are then given by $\bs{\psi}_{i}=(\psi_{i},\psi_{i})$, $i=1,2$,
from which we have
\begin{equation*}
\Psi_{1}^{(0)} = \frac{ J_1(\jp)\cos\tht}{J_0(\jp)\sqrt{\pi(\jp^2-1)}},
\quad\Psi_{2}^{(0)} = \frac{J_1(\jp)\sin\tht}{J_0(\jp)\sqrt{\pi(\jp^2-1)}}
\end{equation*}
and $\Psi_{i}^{(1)} = 0,$ $i=1,2.$ Proceeding as before, we
\begin{equation*}
\Lambda_{11}^{0} = \Lambda_{22}^{0} = \frac{2J_{1}^{2}(\jp)}{  J_0^2(\jp)}=2\jp^2=2\l\quad \text{and} \quad \Lambda_{ij}^{0}=0 \;\;(i\neq j).
\end{equation*}
For the next term we now obtain
\begin{align*}
\L_{ii}^{(1)}  = &
-\frac{\jp^2 }{J_0^2(\jp) (\jp^2-1)}
\int_{0}^{1} \frac{r^3}{1-r^2} \bigg[
J_1^2(\jp r)- J_1^2(\jp) + J_0^2(\jp r) + J_0^2(\jp)
\\
&\hphantom{\frac{\jp^2 }{J_0^2(\jp) (\jp^2-1)}
\int_{0}^{1}}- \frac{2}{\jp r} J_0(\jp r) J_1(\jp r)
\bigg]\di r
-\l\ln 2.
\end{align*}
for $i=1,2$ and $\L_{ij}=0$ for $i\neq j$.

>From this, and again computing the relevant integrals numerically, we obtain
\[
\begin{array}{lll}
\l_i(\e)& = & (j_{1,1}')^2+\e^2\Big(2\l\ln\e+\L^{(1)}_{ii}\Big)+\Odr(\e^{2+\rho})
\eqskip
 & \approx & 3.3900 + 6.7799\, \e^2\ln\e -1.8555\,\e^2 +\Odr(\e^{2+\rho}), \quad i=1,2.
\end{array}
\]
Due to the radial symmetry of $\om$, it is clear that these two eigenvalues should coincide, and the
associate eigenfunctions converge to $\bs{\psi}_1$ and $\bs{\psi}_2$.

\section{Preliminaries}\label{prelim}

In this section we discuss two parameterizations of the surface $\hs_\e$ and prove three auxiliary lemmas which will be used in the next
sections for proving Theorems~\ref{th1.1},~\ref{th2.3}.

\subsection{First parametrization of $\hs_\e$}

The first parametrization is that used in the definition of $\hs_\e$ in (\ref{1.0}), i.e., each point on $\hs_\e$ is described as $x_{n+1}=\pm \e h_\pm(x')$, $x'\in\overline{\om}$, where the sign corresponds to the upper or lower part of $\hs_\e$. Let us first calculate the metrics on $\hs_\e$ in terms of the variables $x'$.

The tangential vectors to
$\hs_\e$ at the point $x'\in\om$, $x_{n+1}=\e h_\pm(x')$ are
\begin{equation*}
\left(0,\ldots,0,1,0,\ldots,0,%D%
\e%D%
\frac{\p h_\pm}{\p x_i}\right),\quad
i=1,\ldots,n,
\end{equation*}
where ``$1$'' stands on $i$-th position. Thus, the metric tensor has
the form
\begin{equation*}
\mathrm{G}_\pm(x',\e):=
\begin{pmatrix}
1+\e^2\left(\frac{\p h_\pm}{\p x_1}\right)^2 & \e^2 \frac{\p
h_\pm}{\p x_1} \frac{\p h_\pm}{\p x_2} & \e^2 \frac{\p h_\pm}{\p
x_1} \frac{\p h_\pm}{\p x_3} & \ldots & \e^2 \frac{\p h_\pm}{\p x_1}
\frac{\p h_\pm}{\p x_n}
\\
\e^2 \frac{\p h_\pm}{\p x_2} \frac{\p h_\pm}{\p x_1} &
 1+\e^2\left(\frac{\p h_\pm}{\p x_2}\right)^2 & \e^2 \frac{\p
h_\pm}{\p x_2} \frac{\p h_\pm}{\p x_3}   & \ldots & \e^2 \frac{\p
h_\pm}{\p x_2} \frac{\p h_\pm}{\p x_n}
\\
\e^2 \frac{\p h_\pm}{\p x_3} \frac{\p h_\pm}{\p x_1} & \e^2 \frac{\p
h_\pm}{\p x_3} \frac{\p h_\pm}{\p x_2} &
 1+\e^2\left(\frac{\p h_\pm}{\p x_3}\right)^2 &   \ldots & \e^2 \frac{\p
h_\pm}{\p x_3} \frac{\p h_\pm}{\p x_n}
\\
\vdots & \vdots & \vdots & \ddots & \vdots
\\
\e^2 \frac{\p h_\pm}{\p x_n} \frac{\p h_\pm}{\p x_1} & \e^2 \frac{\p
h_\pm}{\p x_n} \frac{\p h_\pm}{\p x_2} & \e^2 \frac{\p h_\pm}{\p
x_{n-1}} \frac{\p h_\pm}{\p x_3}
 & \vdots & 1+\e^2\left(\frac{\p h_\pm}{\p x_n}\right)^2
\end{pmatrix}
\end{equation*}
It easy to see that
\begin{equation}\label{3.3}
\mathrm{G}_\pm(x',\e)=\mathrm{E}+\e^2\mathrm{Q}_\pm,\quad
\mathrm{Q}_\pm:=(\nabla_{x'} h_\pm) (\nabla_{x'} h_\pm)^*,
\end{equation}
where $\nabla_{x'} h_\pm$ is treated as a column vector, and ``$*$''
denotes transposition.

\begin{lemma}\label{lm3.2}
The matrix $G_\pm$ has two eigenvalues, the $(n-1)$-multiple
eigenvalue $1$, and the simple eigenvalue $(1+\e^2|\nabla_{x'}
h_\pm|^2)$. The identity
\begin{equation}\label{3.4}
\di \hs_\e=J_\e^\pm \di x', \quad J_\e^\pm:=\sqrt{1+\e^2 |\nabla_{x'}
h_\pm|^2}, \quad \di x'=\di x_1 \di x_2\cdots \di x_n,
\end{equation}
holds true.
\end{lemma}

\begin{proof}
From~(\ref{3.3}) we may write the eigenvalue problem for the matrix $G_{\pm}$ as
\[
\left(\mathrm{E}+\e^{2}v v^{*}
\right)u = z u,
\]
and
\[
(z-1)u=\e^{2}v v^{*}u,
\]
where $v=\nabla_{x'} h_{\pm}$. We thus see that any vector orthogonal to
$v$ is an eigenvector for the above equation with eigenvalue $z$ equal
to one. This yields an eigenvalue of multiplicity $n-1$ if $v$ is not zero,
and $n$ in case $v$ vanishes. In the former case, we easily see that $v$ is
also an eigenvector, now with eigenvalue $1+\e^2 |v|^2$, which will have
multiplicity one. The determinant of $G_{\pm}$ is thus $g^{\pm}=1+\e^2 |v|^2$,
yielding the volume element to be $\sqrt{1+\e^2 |v|^2}$ as desired.
\end{proof}

In what follows we shall make use of the differential expression for the operator $\He$, namely, its expansion w.r.t. $\e$. The expression itself is given by (\ref{2.3a}),
while using (\ref{3.3}) allows us to expand some of the terms in this expression in powers of $\e$,
\begin{align*}
&(\mathrm{E}+\e^2\mathrm{Q}_\pm)^{-1}=\mathrm{E}-\e^2 \mathrm{Q}_\pm+\Odr(\e^4),
\\
&(1+\e^2|\nabla_{x'} h_\pm|^2)^{\pm\frac{1}{2}}=1\pm \e^2\frac{|\nabla_{x'}h_\pm|^2}{2}
+\Odr(\e^4),
\end{align*}
where the plus and minus signs correspond to the upper and lower parts of $S_\e$, respectively. We substitute these formulas into (\ref{2.3a}) and get
\begin{equation}\label{5.1}
\He=-\D_{x'}-\e^2 \left(\frac{|\nabla_{x'}h_\pm|^2}{2}\D_{x'}+ \Div_{x'} \left(\frac{|\nabla_{x'}h_\pm|^2}{2}-\mathrm{Q}_\pm
\right)\nabla_{x'}
\right)+\Odr(\e^4).
\end{equation}

The disadvantage of the parametrization by the variables $x'$ is that the functions $h_\pm$ are not smooth in a vicinity of $\p\om$ and their derivatives blow-up at the boundary $\p\om$.
We shall show it below while introducing the second parametrization. The main idea of the second parametrization is to use special coordinates in a vicinity of $\p\om$ so that they involve  smooth functions only; this parametrization is purely local and will be used only in a vicinity of $\p\om$. It is natural to expect the existence of such coordinates since the surface $\hs_\e$ is infinitely differentiable.

\subsection{Second parametrization of $\hs_\e$}

In a neighborhood of $\p\om$ we introduce new coordinates
$(\tau,s)$, where $s=(s_1,\ldots,s_{n-1})$ are local coordinates on
$\p\om$ corresponding to a $C^\infty$-atlas, and $\tau$, we remind, is the distance to a point measured in the direction of the inward normal $\nu=\nu(s)$ to $\p\om$. Let $r=r(s)$ be the
vector-function describing $\p\om$.   %D%, and $\nu(s)$ be the %inward
%normal to $\p\om$ expressed in terms of $s$.
We have
\begin{equation}\label{3.8}
\begin{aligned}
& x'=r(s)+\tau\nu(s),\quad \nabla_{(\tau,s)}=
\mathrm{M}(\tau,s)\nabla_{x'},
\\
&\mathrm{M}=\mathrm{M}(\tau,s)=
\begin{pmatrix}
\nu
\\
\frac{\p r}{\p s_1}+\tau \frac{\p\nu}{\p s_1}
\\
\ldots
\\
\frac{\p r}{\p s_{n-1}}+\tau \frac{\p\nu}{\p s_{n-1}}
\end{pmatrix},
\end{aligned}
\end{equation}
where $\nu(s)$ and the other vectors in the definition of $M$ are
treated as columns. The vectors $\frac{\p r}{\p s_i}$ are tangential
to $M$ and linear independent, while $\nu(s)$ is orthogonal to
$\p\om$. Thus, the matrix $M$ is invertible for all sufficiently
small $\tau$ and all $s\in\p\om$. The inequalities
\begin{equation}\label{3.9}
C_1\leqslant \mathrm{M}(\tau,s)\leqslant C_2,\quad C_2^{-1}\leqslant
\mathrm{M}^{-1}(\tau,s)\leqslant C_1^{-1},\quad s\in\p\om,\quad
\tau\in[-\tau_0,\tau_0],
\end{equation}
are valid, where $C_1$, $C_2$ are positive constants independent of $(\tau,s)$.
It follows from these estimates and (\ref{3.8}) that the matrix
$M^{-1}(\tau,s)$ is infinitely differentiable in the neighbourhood
$\{x: |\tau|<\tau_0\}$ of $\p\om$.

Consider now equations (\ref{1.2}). By assumption~(\ref{A2}) they
have the smooth solution $\tau=a(x_{n+1},P)$ and, for small
$x_{n+1}$, the function $a$ behaves as
\begin{equation*}
a(x_{n+1},P)=a_2(P)x_{n+1}^2+\Odr(x_{n+1}^3).
\end{equation*}
Hence,
\begin{align}
&h_\pm\big(P+\tau \nu(P)\big)=x_{n+1}=\pm
a_2^{-\frac{1}{2}}(P)\tau^{\frac{1}{2}}+\Odr(\tau),\quad
\tau\to+0,\nonumber
\\
&\nabla_{x'} h_\pm=\mathrm{M}^{-1}\nabla_{(\tau,s)} h_\pm,\nonumber
\\
&C_3\tau^{-1}\leqslant |\nabla_{x'} h_\pm|^2\leqslant C_4
\tau^{-1},\quad \tau\in%D%
(%D%
0,\tau_0],\label{3.12}
\end{align}
where $C_3$, $C_4$ are positive constants independent of $(\tau,s)$. As we see from the last estimates, the functions $h_\pm$ are not smooth at the point $\tau=0$, i.e., at $\p\om$.

 We employ once again assumption
(\ref{A2}) and pass from equations $x_{n+1}=\pm\e h_\pm(x')$ to
\begin{equation}\label{4.8}
\tau=a(t,P),\quad x_{n+1}=\e t,\quad x'=r(s)+\tau \nu(s).
\end{equation}
%D%where, we recall, $\nu$ is the inward normal to $\p\om$. % %$r(s)$ is the vector-function describing $\p\om$ and
%$s$ are local coordinates on $\p\om$.
It follows from (\ref{1.4})
that the function $a(t,P)$ can be represented as $t^2
\widetilde{a}(t,P)$, where $\widetilde{a}\in
C^\infty([-t_0,t_0]\times\p\om)$ and $\widetilde{a}>0$ for
sufficiently small $t_0$.

We introduce a new variable $\z=t \widetilde{a}^{\frac{1}{2}}(t,P)$.
%D%
>From assumption~(\ref{A2}) we conclude that
\begin{equation}\label{4.9a}
t=b(\z,P)\in C^\infty([-\z_0,\z_0]\times\p\om)
\end{equation}
for a fixed small constant $\z_0$, and the Taylor series for $a$ and $b$ read as follows,
\begin{align}
&a(t,P)=\sum\limits_{i=2}^{\infty} a_i(P)t^i,\quad t\to+0,
\label{4.9}
\\
&b(\z,P)=\sum\limits_{i=1}^{\infty} b_i(P)\z^i,\quad \z\to0,
\qquad b_1:=a_2^{-\frac{1}{2}},
\label{4.52}
\end{align}
where $a_i, b_i\in C^\infty(\p\om)$. We define a rescaled variable $\xi:=\z\e^{-1}$.
The final form of the second parametrization for $\hs_\e$ is as follows,
\begin{equation}\label{4.54}
x'=r(s)+\e^2\xi^2 \nu(s),\quad x_{n+1}=\e^2 b_\e(\xi,r(s)),\quad
\xi\in[-\z_0\e^{-1},\z_0\e^{-1}],
\end{equation}
where $b_\e(\xi,P):=\e^{-1}b(\e\xi,P)$ and $\z_0$ is a fixed
sufficiently small number. We observe that by the definition of $\z$
\begin{equation}\label{4.55}
\tau=a(t,P)=\z^2=\e^2\xi^2.
\end{equation}

As in (\ref{5.1}), we shall also employ the expansion in $\e$ of the differential expression for $\He$ corresponding to the second parametrization. We find first the tangential vectors to $S_\e$ corresponding to the
parametrization (\ref{4.54}),
\begin{equation}\label{4.15}
T_{s_i}=\left(\frac{\p r}{\p s_i}+\e^2\xi^2  \frac{\p\nu}{\p
s_i},\e^2\frac{\p b_\e}{\p s_i}\right),\quad T_\xi=\e^2\left(2\xi^2
\nu,\frac{\p b_\e}{\p \xi}\right).
\end{equation}
It is clear that the vectors $\frac{\p
r}{\p s_i}$, $\frac{\p \nu}{\p s_i}$ belong to the
tangential plane and are orthogonal to $\nu$. Employing this fact and (\ref{4.15}), we calculate the
metric tensor,
\begin{align*}
&(T_\xi,T_\xi)_{\mathds{R}^{n+1}}=\e^4\left(4\xi^2+\left(\frac{\p
b_\e}{\p\xi}\right)^2\right),\quad
(T_\xi,T_{s_i})_{\mathds{R}^{n+1}}=\e^4\frac{\p b_\e}{\p\xi}\frac{\p
b_\e}{\p s_i},
\\
&(T_{s_i},T_{s_j})_{\mathds{R}^{n+1}}=\left(\frac{\p r}{\p
s_i}+\e^2\xi^2\frac{\p\nu}{\p s_i}, \frac{\p r}{\p
s_j}+\e^2\xi^2\frac{\p\nu}{\p
s_j}\right)_{\mathds{R}^{n+1}}+\e^4\frac{\p b_\e}{\p s_i}\frac{\p
b_\e}{\p s_j}.
\end{align*}
By Weingarten equations  we see that %D%
\begin{equation*}
\big((T_{s_i},T_{s_j})_{\mathds{R}^{n+1}}\big)_{i,j=\overline{1,n}}
=\mathrm{A},
\end{equation*}
where
\begin{equation}\label{4.18}
\begin{aligned}
\mathrm{A}:=&\mathrm{G}_{\p\om}
-2\e^2\xi^2\mathrm{B}+\e^4\xi^4\mathrm{B}\mathrm{G}_{\p\om}^{-1}\mathrm{B}+
\e^4(\nabla_s b_\e)(\nabla_s b_\e)^*
\\
=&\mathrm{G}_{\p\om}(\mathrm{E}-\e^2\xi^2
\mathrm{G}_{\p\om}^{-1}\mathrm{B})^2+\e^4(\nabla_s b_\e)(\nabla_s
b_\e)^*,
\end{aligned}
\end{equation}
$\mathrm{G}_{\p\om}$ is the metric tensor of $\p\om$ associated with the coordinates $s$, $\mathrm{B}$ is the second fundamental form of $\p\om$ corresponding to the orientation defined by $\nu$. Hence, the metric tensor $\mathrm{G}_\e$ of $S_\e$ associated with the parametrization (\ref{4.54}) reads as follows,
\begin{equation*}
\mathrm{G}_\e=
\begin{pmatrix}
\e^4\Big(4\xi^2+\big(\frac{\p b_\e}{\p\xi}\big)^2\Big) &
\e^4\mathrm{p}^*
\\
\e^4\mathrm{p} & \mathrm{A}
\end{pmatrix},
\quad  \mathrm{p}:= \frac{\p b_\e}{\p \xi}\nabla_s b_\e.
\end{equation*}
By direct calculations we check that
\begin{align}
&\mathrm{G}_\e^{-1}=
\begin{pmatrix}
\e^{-4}\b & -\b \mathrm{p}^*\mathrm{A}^{-1}
\\
\\
-\b \mathrm{A}^{-1} \mathrm{p} & \mathrm{A}^{-1}+\e^4\b
\mathrm{A}^{-1}\mathrm{p} \mathrm{p}^* \mathrm{A}^{-1}
\end{pmatrix}\label{4.17}
\\
&\b:=\left( 4\xi^2+\left(\frac{\p b_\e}{\p\xi}\right)^2
-\e^4\mathrm{p}^* \mathrm{A}^{-1}\mathrm{p}\right)^{-1}.\nonumber
\end{align}
The quantities in (\ref{4.17}) are well-defined provided $\z_0$ is
sufficiently small. Indeed, by (\ref{4.9})
\begin{equation*}
\mathrm{A}=\mathrm{G}_{\p\om}+\Odr(\z^2), \quad
\mathrm{p}=\Odr(1),\quad  \frac{\p b}{\p \z}(\z,P)=\Odr(1),\quad
\z\to0,
\end{equation*}
that implies the existence of $\mathrm{A}^{-1}$ and $\b$. In what
follows we assume that $\z_0$ is chosen in such a way.

By $K_i=K_i(s)$, $i=1,\ldots,n-1$, we denote the principal
curvatures of $\p\om$, and $K:=\sum\limits_{i=1}^{n-1} K_i$. We note that $(n-1)^{-1}K$ is the mean curvature of $\p\om$ and
let
\begin{equation*}
\mathrm{a}:=\Det \big((\mathrm{E}-\e^2\xi^2
\mathrm{G}_{\p\om}^{-1}\mathrm{B})^2+\e^4\mathrm{G}_{\p\om}^{-1}(\nabla_s
b_\e)(\nabla_s b_\e)^*\big).
\end{equation*}

\begin{lemma}\label{lm4.1}
The identities
\begin{align}
&
\begin{aligned}
b_\e=\sum\limits_{i=1}^{\infty} b_i(P)\e^{i-1}\xi^i, \quad
\mathrm{A}^{-1}=\mathrm{G}_{\p\om}^{-1}+\Odr(\e^2\xi^2),\quad
\mathrm{p}=\xi b_1\nabla_s b_1 +\Odr(\e\xi^2),
\end{aligned}\label{4.36}
\\
&\Det \mathrm{G}_\e=\e^4\b^{-1}\Det \mathrm{A},\label{4.28a}
\\
&
\begin{aligned}
& \Det \mathrm{A}=\mathrm{a}\Det  \mathrm{G}_{\p\om} ,\quad
\mathrm{a}=\sum\limits_{i=0}^{2} \e^{2i}\a_{2i}
+\Odr(\e^4\xi^4),
\end{aligned}
\label{4.28b}
\\
&\a_0:=1,
\quad\a_2:=-2\xi^2 K.
\label{4.27a}
\end{align}
hold
true.
\end{lemma}

\begin{proof}
The identities (\ref{4.36}) follow directly from the definition of
$b_\e$, $\mathrm{A}$, and $\mathrm{p}$.

We make linear transformations in (\ref{4.17}) to calculate the
determinant of $\mathrm{G}_\e$,
\begin{align*}
(\Det \mathrm{G}_\e)^{-1}=\Det^{-1} \mathrm{\mathrm{G}}_\e=
\begin{vmatrix}
\e^{-4}\b & -\b \mathrm{p}^* \mathrm{A}^{-1}
\\
0 & \mathrm{A}^{-1}
\end{vmatrix}=\e^{-4}\b\Det^{-1}\mathrm{A}
\end{align*}
that proves (\ref{4.28a}).

It is  easy to see that
\begin{equation}\label{4.27}
\Det \mathrm{A}=\mathrm{a}\Det \mathrm{G}_{\p\om}.
\end{equation}
In view of (\ref{4.18}) we get %\corr %D%obtain
\begin{align*}
\mathrm{a} &= \Det\big(\mathrm{E} +\e^4
(\mathrm{E}-\e^2\xi^2 \mathrm{G}_{\p\om}^{-1}\mathrm{B})^{-2}
\mathrm{G}_{\p\om}^{-1}(\nabla_s b_\e)(\nabla_s b_\e)^*\big) \Det (\mathrm{E}-\e^2\xi^2
\mathrm{G}_{\p\om}^{-1}\mathrm{B})^2
\\
&=\big(1+\e^4\Tr  (\mathrm{E}-\e^2\xi^2
\mathrm{G}_{\p\om}^{-1}\mathrm{B})^{-2} \mathrm{G}_{\p\om}^{-1}
(\nabla_s b_\e)(\nabla_s b_\e)^*  +\Odr(\e^8\xi^2)\big)
\prod\limits_{i=1}^{n-1}(1-\e^2\xi^2 K_i)^2
\\
&=\big(1+\e^4\Tr \mathrm{G}_{\p\om}^{-1} (\nabla_s b_\e)(\nabla_s
b_\e)^*+\Odr(\e^6\xi^4)\big)\big( 1-2\e^2\xi^2 K +\Odr(\e^4\xi^4)\big)
\\
&= \big(1+\e^4 |\nabla b_\e|^2+\Odr(\e^6\xi^4)\big)\big(
1-2\e^2\xi^2 K+\Odr(\e^4\xi^4)\big).
\end{align*}
We substitute the obtained formula and (\ref{4.52}) into
(\ref{4.27}) and arrive at (\ref{4.28b}).
\end{proof}

Employing (\ref{4.18}), (\ref{4.36}), by direct calculations we
check
\begin{align*}
\mathrm{p}^* \mathrm{A}^{-1} \mathrm{p}=& \left( \frac{\p
b_\e}{\p\xi}\right)^2 (\nabla_s b_\e)^* \mathrm{G}_{\p\om}^{-1}
(\nabla_s b_\e)+\Odr(\e^2\xi^2)
\\
=&\left( \frac{\p b_\e}{\p\xi}\right)^2| \nabla
b_\e|^2+\Odr(\e^2\xi^2)
\\
=&b_1^2\xi^2|\nabla b_1|^2+\Odr(\e\xi^2).
\end{align*}

Hence, by (\ref{4.28a}), (\ref{4.28b}) and the definition of $\b$
\begin{align*}
&\e^{-2}\Det^{\frac{1}{2}} \mathrm{G}_\e=
\b^{-\frac{1}{2}}\Det^{\frac{1}{2}}\mathrm{A}= \b^{-1}\b_\mathrm{A} \Det^{\frac{1}{2}}
\mathrm{G}_{\p\om},
\\
&\b_\mathrm{A}:=\b^{\frac{1}{2}}\mathrm{a}^{\frac{1}{2}}=\sum\limits_{i=0}^{4}\e^i
\b_{i-4}+\Odr\big(\e^5(|\xi|^2+\xi^4)\big),
\end{align*}
where $\b_i=\b_i(\xi,P)\in C^\infty(\mathds{R}\times\p\om)$ are some functions. In particular,
\begin{equation}
\begin{aligned}
\b_{-4}:=&\frac{1}{(4\xi^2+ b_1^2)^{\frac{1}{2}}}, \quad \b_{-3}:=
-\frac{2 b_1 b_2\xi}{(4\xi^2+ b_1^2)^{\frac{3}{2}}},
\\
\b_{-2}:=& -\frac{3 b_1 b_3 \xi^2}{(4\xi^2+ b_1^2)^{\frac{3}{2}}}
-\frac{4\xi^2(2\xi^2-b_1^2)b_2^2}{(4\xi^2+b_1^2)^{\frac{5}{2}}}
-\frac{\xi^2 K}{(4\xi^2+b_1^2)^{\frac{1}{2}}},
\end{aligned}\label{4.41}
\end{equation}
while the function $\b_{-1}$, $\b_0$ satisfy the uniform in $\xi$
and $P$ estimates
\begin{equation*}
|\b_{-1}|\leqslant \frac{C|\xi|^3}{1+|\xi|^3},\quad |\b_0|\leqslant
C\xi^2(1+|\xi|).
\end{equation*}
The obtained formulas, Lemma~\ref{lm4.1}, and (\ref{4.17}) allow us to write the expansion for $\mathrm{G}_\e^{-1}$,
\begin{align}
&\e^{-2}(\Det^{\frac{1}{2}} \mathrm{G}_\e)\mathrm{G}_\e^{-1}=
\Det^{\frac{1}{2}}\mathrm{G}_{\p\om}
\sum\limits_{i=-4}^{0}\e^i\mathrm{G}_i +\Odr(\e), \label{4.37}
\\
&
\begin{aligned}
&\mathrm{G}_{i}:=
\begin{pmatrix}
\b_i & 0
\\
0 & 0
\end{pmatrix},\quad i=-4,\ldots,-1,
\\
&\mathrm{G}_0:=
\begin{pmatrix}
\b_0 & - b_1 \xi  \b_{-4} (\nabla_s b_1)^* \mathrm{G}_{\p\om}^{-1}
\\
-b_1 \xi  \b_{-4}  \mathrm{G}_{\p\om}^{-1} \nabla_s b_1 &
\b_{-4}^{-1} \mathrm{G}_{\p\om}^{-1}
\end{pmatrix}.
\end{aligned}\label{4.38}
\end{align}
Taking into account (\ref{4.28a}), (\ref{4.28b}), we write the
operator $\mathcal{H}_\e$ in terms of the variables $(s_0,s)$, where
$s_0:=\xi$,
\begin{equation}\label{4.37a}
\begin{aligned}
\mathcal{H}_\e=&-\frac{1}{\Det^{\frac{1}{2}}\mathrm{G}_\e}
\sum\limits_{i,j=0}^{n-1} \frac{\p}{\p s_i} G_\e^{ij}
\Det^{\frac{1}{2}} \mathrm{G}_\e\frac{\p}{\p s_j}
\\
=&-\frac{\e^{-2}\b_\mathrm{A}}{\mathrm{a} \Det^{\frac{1}{2}}\mathrm{G}_{\p\om}
} \sum\limits_{i,j=0}^{n-1} \frac{\p}{\p s_i} G_\e^{ij}
\Det^{\frac{1}{2}} \mathrm{G}_\e\frac{\p}{\p s_j},
\end{aligned}
\end{equation}
and $G_\e^{ij}$ are the entries of the inverse matrix (\ref{4.17}).
It follows from the last formula and (\ref{4.17}) that
\begin{equation*}
\mathcal{H}_\e=\e^{-4}\mathrm{a}^{-1} \b_\mathrm{A} \frac{\p}{\p\xi}
\b_\mathrm{A} \frac{\p}{\p\xi}+\Odr(1).
\end{equation*}
We employ the obtained equation, (\ref{4.37a}), (\ref{4.37}) and
(\ref{4.38}), and expand the coefficients of $\mathcal{H}_\e$ in
powers of $\e$ leading us to the identities
\begin{align}
&\mathcal{H}_\e= \sum\limits_{i=-4}^0 \e^i \mathcal{L}_i+\Odr(\e),
\label{4.32}
\\
&
\begin{aligned}
& \mathcal{L}_{-4}:=\mathcal{L}^{(-4)}, \quad
\mathcal{L}_{-3}:=\mathcal{L}^{(-3)}, \quad
\mathcal{L}_{-2}:=\mathcal{L}^{(-2)}+\a^{(2)} \mathcal{L}^{(-4)},
\\
&\mathcal{L}_{-1}:=\mathcal{L}^{(-1)}+\a^{(2)} \mathcal{L}^{(-3)},
\quad \mathcal{L}_{0}:=\mathcal{L}^{(0)}+\a^{(2)}
\mathcal{L}^{(-2)}+\a^{(4)}\mathcal{L}^{(-4)},
\\
&\a^{(2)}:=2\xi^2 K,\quad \a^{(4)}=\a^{(4)}(\xi,s),
\end{aligned}\label{4.39a}
\\
& \mathcal{L}^{(i)}:=-\sum\limits_{j=0}^{i+4} \b_{j-4}
\frac{\p}{\p\xi} \b_{i-j} \frac{\p}{\p\xi},\quad
i=-4,\ldots,-1,\label{4.40a}
\\
&
\begin{aligned}
\mathcal{L}^{(0)}:=& -\sum\limits_{l=0}^{4} \b_{l-4}
\frac{\p}{\p\xi} \b_{-l} \frac{\p}{\p\xi} +  b_1\b_{-4}\frac{\p}{\p\xi}
\xi\b_{-4} (\nabla_s b_1)^* \mathrm{G}_{\p\om}^{-1}\nabla_s
\\
&+ \b_{-4}\Det^{-\frac{1}{2}} \mathrm{G}_{\p\om} \Div_s
b_1\b_{-4}\xi \Det^{\frac{1}{2}} \mathrm{G}_{\p\om} (\nabla_s b_1)^*
\mathrm{G}_{\p\om}^{-1} \frac{\p}{\p\xi}
\\
&- \b_{-4}\Det^{-\frac{1}{2}} \mathrm{G}_{\p\om} \Div_s\b_{-4}^{-1}
(\Det^{\frac{1}{2}} \mathrm{G}_{\p\om}) \mathrm{G}_{\p\om}^{-1}
\nabla_s.
\end{aligned}  \label{4.41b}
\end{align}

\subsection{Auxiliary lemmas}

We proceed to the auxiliary lemmas which will be used for proving Theorem~\ref{th2.3}.

\begin{lemma}\label{lm4.2}
In a vicinity of $\p\om$ the identities
\begin{align}
&\Det \mathrm{M}=(\Det^{\frac{1}{2}} \mathrm{G}_{\p\om})\prod\limits_{i=1}^{n-1} (1-\tau K_i), \label{4.23}
\\
&-\D_{x'}=-\frac{1}{\Det \mathrm{M}} \Div_{(\tau,s)} (\Det \mathrm{M}) \widehat{\mathrm{M}} \nabla_{(\tau,s)} \nonumber
\end{align}
hold true, where
\begin{equation}\label{4.29a}
\widehat{\mathrm{M}}:=(\mathrm{M}^{-1})^*  \mathrm{M}^{-1}=
\begin{pmatrix}
1 & 0
\\
0 & (\mathrm{E}-\tau \mathrm{G}^{-1}_{\p\om}\mathrm{B})^{-2}\mathrm{G}^{-1}_{\p\om}
\end{pmatrix}.
\end{equation}
\end{lemma}

\begin{proof}
It follows from (\ref{3.8}) and the Weingarten formulas that
\begin{equation*}
\mathrm{M}=
\begin{pmatrix}
\nu
\\
\frac{\p r}{\p s_i}-\tau\sum\limits_{k=1}^{n-1} B_i^k \frac{\p r}{\p
s_k}
\end{pmatrix},
\end{equation*}
where $B_i^k$ are the entries of the matrix
$\mathrm{G}_{\p\om}^{-1}\mathrm{B}$, and all vectors are treated as
rows.

A straightforward direct calculation allows us to check that the inverse matrix
$\mathrm{M}^{-1}$ reads as follows,
\begin{equation}\label{4.30}
\mathrm{M}^{-1}=
\begin{pmatrix}
\nu
\\
\sum\limits_{k=1}^{n-1} c_i^k \frac{\p r}{\p s_k}
\end{pmatrix}^*,
\end{equation}
where, as before, $^*$ indicates matrix transposition, and
$c_i^k$ are the entries of the matrix $\mathrm{C}=(\mathrm{E}-\tau
\mathrm{G}_{\p\om}^{-1} \mathrm{B})^{-1}\mathrm{G}_{\p\om}^{-1}$.

Let $u_1, u_2\in C_0^\infty(\om)$ be any two functions with the corresponding
supports located in a neighbourhood of $\p\om$, where the coordinates
$(\tau,s)$ are well-defined. We integrate by parts,
\begin{align*}
(-\D_{x'}u,v)_{L_2(\om)}&=(\nabla_{x'}u,\nabla_{x'}v)_{L_2(\om)} =
(\mathrm{M}^{-1}\nabla_{(\tau,s)}u,(\Det \mathrm{M})
\mathrm{M}^{-1}\nabla_{(\tau,s)}v )_{L_2((0,\tau_0)\times\p\om)}
\\
&=\big(-\Div_{(\tau,s)}(\Det \mathrm{M})(\mathrm{M}^{-1})^*
(\mathrm{M}^{-1})\nabla_{(\tau,s)}u, v\big)_{L_2((
0,\tau_0)\times\p\om)}
\\
&= \big(-(\Det^{-1}\mathrm{M})\Div_{(\tau,s)}(\Det
\mathrm{M})(\mathrm{M}^{-1})^*  \mathrm{M}^{-1} \nabla_{(\tau,s)}u,
v\big)_{L_2(\om)}.
\end{align*}
Hence,
\begin{equation}\label{4.34}
-\D_{x'}=- (\Det^{-1}\mathrm{M})\Div_{(\tau,s)}(\Det
\mathrm{M})(\mathrm{M}^{-1})^*  \mathrm{M}^{-1} \nabla_{(\tau,s)}.
\end{equation}
In view of (\ref{4.30}) we have
\begin{align*}
(\mathrm{M}^{-1})^*  \mathrm{M}^{-1}&=
\begin{pmatrix} \nu
\\
\sum\limits_{k=1}^{n-1} c_i^k \frac{\p r}{\p s_k}
\end{pmatrix}
\begin{pmatrix}
\nu
\\
\sum\limits_{k=1}^{n-1} c_i^k \frac{\p r}{\p s_k}
\end{pmatrix}^*=
\begin{pmatrix}
1 & 0
\\
0 & \mathrm{C} \mathrm{G}_{\p\om} \mathrm{C}
\end{pmatrix}
\\
&=
\begin{pmatrix}
1 & 0
\\
0 & (\mathrm{E}-\tau \mathrm{G}^{-1}_{\p\om}\mathrm{B})^{-2}\mathrm{G}^{-1}_{\p\om}
\end{pmatrix},
\\
\Det^{-2} \mathrm{M} =&\Det (\mathrm{M}^{-1})^* \mathrm{M}^{-1} =
\Det (\mathrm{E}-\tau \mathrm{G}^{-1}_{\p\om}\mathrm{B})^{-2}\Det \mathrm{G}_{\p\om}^{-1},
\\
\Det \mathrm{M}=&\Det^{\frac{1}{2}} \mathrm{G}_{\p\om} \Det (\mathrm{E}-\tau \mathrm{G}_{\p\om}^{-1}
\mathrm{B}) =\Det^{\frac{1}{2}}\mathrm{G}_{\p\om}\prod\limits_{i=1}^{n-1}(1-\tau K_i).
\end{align*}
The obtained formulas and (\ref{4.34}) imply the statement of the
lemma.
\end{proof}

We recall that the set $\om^\d$ was introduced in (\ref{3.12a}).

\begin{lemma}\label{lm4.3}
Let the functions $f_\pm\in C^\infty(\om_\pm)$ satisfy the differentiable asymptotics
\begin{equation}\label{4.78}
f_\pm(x')=\sum\limits_{j=-4}^{\infty} f_{j/2}^\pm(P)\tau^{\frac{j}{2}}, \quad\tau\to+\infty,
\end{equation}
uniformly in $P\in\p\om_\pm$, where $f_{j/2}^\pm\in C^\infty(\p\om_\pm)$, and $V^{(0)}, V^{(1)}\in C^\infty(\p\om)$ are some functions.   Suppose the condition
\begin{equation}\label{4.79}
\begin{aligned}
\lim\limits_{\d\to+0} \Bigg[& (f_+,\psi_+^{(i)})_{L_2(\om^\d)}+ (f_-,\psi_-^{(i)})_{L_2(\om^\d)} - \d^{-1} \int\limits_{\p\om} (f_{-2}^+ + f_{-2}^-)\Psi_i^{(0)}\di s
\\
&- 2\d^{-1/2} \int\limits_{\p\om} (f_{-3/2}^+ + f_{-3/2}^-) \Psi_i^{(0)}\di s
\\
&-\ln \d \int\limits_{\p\om} \Big((K (f_{-2}^+ + f_{-2}^-) - f_{-1}^+ - f_{-1}^-\big)\Psi_i^{(0)} - (f_{-2}^+ - f_{-2}^-)\Psi_i^{(1)}\Big)\di s\Bigg]
\\
&-\int\limits_{\p\om}(f_{-2}^+ - f_{-2}^-) \Psi_i^{(1)}\di s + \int\limits_{\p\om} (f_{-2}^+  + f_{-2}^-) \Psi_i^{(0)} K \di s
\\
&+2\int\limits_{\p\om} \big(V^{(0)} \Psi_i^{(1)}-V^{(1)}\Psi_i^{(0)}\big)\di s=0,\quad i=1,\ldots,m,
\end{aligned}
\end{equation}
holds true. Then there exist the unique solutions $u_\pm\in C^\infty(\om_\pm)$ to the equations
\begin{equation}\label{4.80}
(-\D_{x'}-\l)u_\pm=f_\pm,\quad x\in\om_\pm,
\end{equation}
these solutions satisfy differentiable asymptotics
\begin{equation}\label{4.81}
\begin{aligned}
u_\pm(x')=&f_{-2}^\pm(P)\ln\tau+ U^{(0)}(P)\pm V^{(0)}(P) + 4 f_{-3/2}^\pm(P) \tau^{1/2} +\tau(V^{(1)}(P)\pm U^{(1)}(P))
\\
& + \tau (1-\ln\tau) \big(f_{-1}^\pm(P)-K(P) f_{-2}^\pm(P) \big)
+ \Odr(\tau^{3/2}),\quad \tau\to0,
\end{aligned}
\end{equation}
uniformly in $P\in\p\om_\pm$, where $U^{(0)}, U^{(1)}\in C^\infty(\p\om_\pm)$ are some functions, and the condition
\begin{equation}\label{4.81a}
(U_0,\Psi_i^{(0)})_{L_2(\p\om)}+ (U_1,\Psi_i^{(1)})_{L_2(\p\om)}=0,\quad i=1,\ldots,m,
\end{equation}
holds true.
\end{lemma}

\begin{proof}
Let $\chi(\tau)$ be the cut-off function introduced in the proof of Lemma~\ref{lm3.3}. We introduce the functions
\begin{align*}
\widehat{u}_\pm(x'):=\Bigg(&f_{-2}^\pm(P)\ln\tau \pm V^{(0)}(P) + 4 f_{-3/2}^\pm(P) \tau^{1/2}
\\
&+ \tau (1-\ln\tau) \big(f_{-1}^\pm(P)-K(P) f_{-2}^\pm(P) \big)
\\
&+\tau V^{(1)}(P)-\frac{4}{3}\tau^{3/2} \big(f_{-1/2}^\pm(P)-2K(P)f_{-3/2}^\pm(P)\big) \Bigg) \chi(\tau).
\end{align*}
Employing Lemma~\ref{lm4.2}, one can check that
\begin{equation}\label{4.82}
(-\D_{x'}-\l)\widehat{u}_\pm(x')=\chi(\tau) \sum\limits_{j=-4}^{-1} f_{j/2}^\pm(P)\tau^j + \widehat{f}_\pm(x'),
\end{equation}
where $\widehat{f}_\pm\in C^\infty(\om_\pm)\cap L_2(\om_\pm)$.

We construct the solutions to (\ref{4.80}) as
\begin{equation*}
u_\pm=\widehat{u}_\pm+\widetilde{u}_\pm.
\end{equation*}
Substituting this identity and (\ref{4.82}) into (\ref{4.80}), we obtain the equations for $\widetilde{u}_\pm$,
\begin{equation}\label{4.84}
(-\D_{x'}-\l)\widetilde{u}_\pm=\widetilde{f}_\pm, \quad \widetilde{f}_\pm:=f_\pm-\chi\sum\limits_{j=-4}^{-1} f_{j/2}^\pm \tau^j - \widehat{f}_\pm,
\end{equation}
and by (\ref{4.78}) we have $\widetilde{f}_\pm\in L_2(\om_\pm)$. Hence, we can rewrite these equations as
\begin{equation}\label{4.85}
(\mathcal{H}_0-\l)\widetilde{\bs{u}}=\widetilde{\bs{f}},\quad \widetilde{\bs{u}}:=(\widetilde{u}_+,\widetilde{u}_-), \quad \widetilde{\bs{f}}:=(\widetilde{f}_+,\widetilde{f}_-).
\end{equation}
Since $\l$ is a discrete eigenvalue of $\mathcal{H}_0$, the solvability condition of the last equation is
\begin{equation*}
(\widetilde{\bs{f}},\bs{\psi}_i)_{L_2(\bs{\om})}=0,\quad k=1,\ldots,m,
\end{equation*}
which can be rewritten as
\begin{equation*}
(\widetilde{f}_+,\psi_+^{(i)})_{L_2(\om)}+ (\widetilde{f}_+,\psi_+^{(i)})_{L_2(\om)}  =0,\quad k=1,\ldots,m,
\end{equation*}
or, equivalently,
\begin{equation}
\lim\limits_{\d\to0} \Big( (\widetilde{f}_+,\psi_+^{(i)})_{L_2(\om^\d)}+
(\widetilde{f}_-,\psi_-^{(i)})_{L_2(\om^\d)}\Big)=0,\quad k=1,\ldots,m.\label{4.86}
\end{equation}
Integrating by parts and taking into account (\ref{4.82}), (\ref{4.84}),  we get
\begin{align*}
(\widetilde{f}_\pm,\psi_\pm^{(i)})_{L_2(\om^\d)}=& \big(f_\pm+(\D_{x'}+\l)\widetilde{u}_\pm, \psi_\pm^{(i)}
\big)_{L_2(\om^\d)}
\\
= &(f_\pm, \psi_\pm^{(i)})_{L_2(\om^\d)}
-\int\limits_{\p\om^\d} \left( \psi_\pm^{(i)} \frac{\p \widetilde{u}_\pm}{\p\tau}-\widetilde{u}_\pm \frac{\p\psi_\pm^{(i)}}{\p\tau} \right)\di s.
\end{align*}
Here we have used that the normal derivative on $\p\om^\d$ is that w.r.t. to $\tau$ up to the sign. We parameterize the points of $\p\om^\d$ by those on $\p\om$ via the relation $x'=r(s)+\d\nu(s)$. In view of (\ref{3.8}) and (\ref{4.23}) we have
\begin{equation}\label{4.21}
\int\limits_{\p\om^\d} \cdot \di s=\int\limits_{\p\om} \cdot \prod\limits_{j=1}^{n-1}(1-\tau K_j)\di s.
\end{equation}
Taking this formula into account, we continue the calculations,
\begin{align*}
(\widetilde{f}_\pm,&\psi_\pm^{(i)})_{L_2(\om^\d)}=(f_\pm, \psi_\pm^{(i)})_{L_2(\om^\d)}
\\
&-\int\limits_{\p\om} \left( \psi_\pm^{(i)} \frac{\p \widetilde{u}_\pm}{\p\tau}-\widetilde{u}_\pm \frac{\p\psi_\pm^{(i)}}{\p\tau} \right)\Bigg|_{x'=r(s)+\d\nu(s)} \prod\limits_{j=1}^{n-1}(1-\tau K_j) \di s
\\
=&(f_\pm, \psi_\pm^{(i)})_{L_2(\om^\d)}-\d^{-1}\int\limits_{\p\om} f_{-2}^\pm \Psi_k^{(0)}\di s - 2 \d^{-1/2} \int\limits_{\p\om} f_{-3/2}^\pm \Psi_k^{(0)}\di s
\\
&-\ln\d \int\limits_{\p\om} \left( (K f_{-2}^\pm - f_{-1}^\pm) \Psi_i^{(0)}\mp f_{-2}^\pm \Psi_i^{(1)}\right)  \di s
\\
&+\int\limits_{\p\om} f_{-2}^\pm \big(\Psi_i^{(0)} K \mp \Psi_i^{(1)}\big)\di s
+\int\limits_{\p\om} \big(V^{(0)} \Psi_i^{(1)}-V^{(1)}\Psi_i^{(0)}\big)\di s
+\Odr(\d^{1/2}).
\end{align*}
We substitute the last identities into (\ref{4.86}) and arrive at (\ref{4.79}). Thus, the condition (\ref{4.79}) imply the existence of solutions to (\ref{4.80}).

The functions $\widetilde{u}_\pm\in\H^2(\om_\pm)$ satisfy (\ref{1.6}) in the sense of traces. Denote
\begin{equation*}
U^{(0)}:=\widetilde{u}_\pm|_{\p\om},\quad U^{(1)}:=\frac{\p \widetilde{u}_\pm}{\p\tau}\Big|_{\p\om},\quad U^{(0)}, U^{(1)}\in L_2(\p\om).
\end{equation*}
The solution to (\ref{4.85}) is defined up to a linear combination of the eigenfunctions. In  view of the
belongings $U^{(0)}, U^{(1)}\in L_2(\p\om)$ we can choose the mentioned linear combination of the eigenfunctions so
that the condition (\ref{4.81a}) is satisfied. Then the solution to (\ref{4.85}) is unique and the same
is obviously true for (\ref{4.80}). To prove the asymptotics (\ref{4.81}) it is sufficient to study the
smoothness of $\widetilde{u}_\pm$ at $\p\om$.

By standard smoothness improving theorems we conclude that $\widetilde{u}_\pm\in C^\infty(\om)$. Moreover, given any $N>0$, it is easy to construct the function
$\widehat{u}_\pm^{(N)}$ similar to $\widehat{u}_\pm$ such that
\begin{gather*}
\widehat{u}_\pm^{(N)}(x')=\widehat{u}_\pm(x')+\Odr(\tau^2),\quad \tau\to0,
\\
(-\D_{x'}-\l)\widehat{u}_\pm^{(N)}(x')=\chi(\tau) \sum\limits_{j=-4}^{N} f_{j/2}^\pm(P)\tau^j + \widehat{f}_\pm^{(N)}(x'),
\end{gather*}
where $\widehat{f}_\pm^{(N)}\in C^\infty(\om_\pm)\cap C^{N_1}(\overline{\om}_\pm)$, and $N_1=N_1(N)\to+\infty$, $N\to+\infty$. Then, proceeding as above, we can construct the solutions to (\ref{4.80}) as $u_\pm=\widetilde{u}_\pm+\widehat{u}_\pm$, where
$\widetilde{\bs{u}}^{(N)}:=(\widetilde{u}_+^{(N)},\widetilde{u}_-^{(N)})$ solves the equation
\begin{gather*}
(\mathcal{H}_0-\l)\widetilde{\bs{u}}^{(N)}=\widetilde{\bs{f}}^{(N)},\quad \widetilde{\bs{f}}^{(N)}:=(\widetilde{f}_+^{(N)},\widetilde{f}_-^{(N)}),
 \\ \widetilde{f}_\pm^{(N)}(x'):=f_\pm(x')-\chi(\tau)\sum\limits_{j=-4}^N f_{j/2}^\pm(P)\tau^j- \widehat{f}_\pm^{(N)}.
\end{gather*}
It is clear that $\widetilde{f}_\pm^{(N)}$ belongs to $C^{N_2}(\overline{\om}_\pm)$, where $N_2=N_2(N)\to+\infty$ as $N\to+\infty$. Hence, by the smoothness improving theorems $\widetilde{u}_\pm^{(N)}\in C^{N_3}(\overline{\om}_\pm)$, $N_3=N_3(N)\to+\infty$, $N\to+\infty$. Choosing $N$ large enough, we arrive at the asymptotics (\ref{4.81}).
\end{proof}

\begin{lemma}\label{lm4.4}
For all $u,v\in C^\infty(\overline{\om})$ in a small vicinity of $\p\om$ the identities
\begin{align}
&\Div_{x'} \mathrm{Q}_\pm\nabla_{x'} u= \frac{1}{\det \mathrm{M}} \Div_{(\tau,s)} (\det \mathrm{M}) \widehat{\mathrm{M}} \nabla_{(\tau,s)} h_\pm (\nabla_{(\tau,s)} h_\pm)^* \widehat{\mathrm{M}} \nabla_{(\tau,s)} u,\label{4.43}
\\
&(\nabla_{x'} u,\nabla_{x'} v)_{\mathds{R}^d}= \frac{\p u}{\p\tau}\frac{\p v}{\p\tau}
+\nabla u\cdot(\mathrm{E}-\tau \mathrm{B}\mathrm{G}^{-1}_{\p\om})^{-2}\nabla v
\label{4.44}
\end{align}
hold true.
\end{lemma}

\begin{proof}
Let $u,v\in C^\infty(\overline{\om})$ be two arbitrary functions with supports in a small vicinity $\{x': 0\leqslant \tau<\tau_0\}$, where $\tau_0$ is a small fixed number. We choose $\tau_0$ so that in this vicinity the coordinates $(\tau,s)$ are well-defined.

Taking (\ref{3.3}) and (\ref{3.8}) into account, we pass to the variables $(\tau,s)$ and integrate by parts to
obtain
\begin{align*}
&\int\limits_{\om} v\Div_{x'} \mathrm{Q}_\pm\nabla_{x'} u\di x'=- \int\limits_{\om} (\nabla_{x'} v, \nabla_{x'} h_\pm (\nabla_{x'} h_\pm)^* \nabla_{x'} u)_{\mathds{R}^n}\di x'
\\
&=-\int\limits_{[0,\tau_0)\times\p\om} \big(\mathrm{M}^{-1} \nabla_{(\tau,s)} v, \mathrm{M}^{-1} \nabla_{(\tau,s)} h_\pm (\nabla_{(\tau,s)} h_\pm)^*  \widehat{\mathrm{M}}  \nabla_{(\tau,s)} u\big)_{\mathds{R}^n} (\det \mathrm{M}) \di \tau\di s
\\
&=\int\limits_{[0,\tau_0)\times\p\om}  v \Div_{(\tau,s)} (\det \mathrm{M}) \widehat{\mathrm{M}}  \nabla_{(\tau,s)} h_\pm (\nabla_{(\tau,s)} h_\pm)^*  \widehat{\mathrm{M}}  \nabla_{(\tau,s)} u \di \tau\di s
\\
&=\int\limits_{\om}  v (\det{\!}^{-1} \mathrm{M})\Div_{(\tau,s)} (\det \mathrm{M}) \widehat{\mathrm{M}}  \nabla_{(\tau,s)} h_\pm (\nabla_{(\tau,s)} h_\pm)^*  \widehat{\mathrm{M}}  \nabla_{(\tau,s)} u \di x',
\end{align*}
which proves (\ref{4.43}).

The identity (\ref{4.44}) follows from (\ref{3.8}) and (\ref{4.29a}),
\begin{align*}
(\nabla_{x'} u,\nabla_{x'} v)_{\mathds{R}^n}&= (\mathrm{M}^{-1}\nabla_{x'}u, \mathrm{M}^{-1}\nabla_{x'}v)_{\mathds{R}^n}
= (\nabla_{x'}u, \widehat{\mathrm{M}}\nabla_{x'}v)_{\mathds{R}^n}
\\
&= \frac{\p u}{\p\tau}\frac{\p v}{\p\tau}   + \big(\nabla_s u, (\mathrm{E}-\tau \mathrm{G}^{-1}_{\p\om}\mathrm{B})^{-2} \mathrm{G}^{-1}_{\p\om}\nabla_s u\big)_{\mathds{R}^n}
\\
&= \frac{\p u}{\p\tau}\frac{\p v}{\p\tau}   + \nabla u\cdot (\mathrm{E}-\tau \mathrm{B}\mathrm{G}^{-1}_{\p\om})^{-2} \nabla v.
\end{align*}
\end{proof}

\section{Uniform resolvent convergence\label{unifresconv}}

In this section we prove Theorem~\ref{th1.1}. We begin with two auxiliary lemmas.

\begin{lemma}\label{lm1.1}
The identity $\Dom(\Ho)=\H^2(\bs{\om})$ holds true and for each
$\bs{u}\in\Dom(\Ho)$ the operator $\Ho$ acts as
$\Ho(\bs{u})=(-\D_{x'}u_+,-\D_{x'}u_-)$. For each $z\in\mathds{C}\setminus\mathds{R}$ the estimate %\corr\corr
\begin{equation}\label{2.7a}
\|(\mathcal{H}_0-z)^{-1}\|_{L_2(\bs{\om})\to\H^2(\bs{\om})}\leqslant \frac{\ds C}{\ds |%D%
\IM %D%
(z)|}
\end{equation}
holds for some constant $C$, where %\corr %D%
$\IM(z)$ denotes the imaginary %D%
part of $z$.
\end{lemma}
\begin{proof}
The first part follows from the definitions and the considerations above for the space $\H^2(\bs{\om})$.
The second part of the statement follows from the fact that the operator $\mathcal{H}_0$ is self--adjoint with
compact resolvent.
\end{proof}

The description of the spectrum of $\Ho$ as being made up of the union of the Dirichlet and Neumann spectra,
is given in the following lemma, together with some properties which will be useful in the sequel.
\begin{lemma}\label{lm1.2}
The spectrum of $\Ho$ coincides with the union of spectra of $\Hd$ and $\Hn$
counting multiplicities. Namely, if $\l$ is an $m^{(D)}$-multiple eigenvalue
of $\Hd$ with the associated eigenfunctions $\psi_i^{(D)}$, $i=1,\ldots,m^{(D)}$,
and is an $m^{(N)}$-multiple eigenvalue of $\Hn$ with the associated eigenfunctions
$\psi_i^{(N)}$, $i=1,\ldots,m^{(N)}$, then $\l$ is $m^{(D)}+m^{(N)}$-multiple
eigenvalue of $\Ho$ with the associated eigenfunctions
$\bs{\psi}_i=(\psi_i^{(D)},-\psi_i^{(D)})$ and $\bs{\psi}_i=(\psi_i^{(N)},\psi_i^{(N)})$.
For any eigenfunction $\bs\psi=(\psi_+,\psi_-)$ of $\Ho$ we have
$\psi_\pm\in C^\infty(\overline{\om})$ and the asymptotics
\begin{equation*}
\psi_\pm(x')=\Psi^{(0)}(P)\pm\tau \Psi^{(1)}(P)+\Odr(\tau^2),\quad P\in\p\om,
\end{equation*}
where
\begin{align*}
&\Psi^{(0)}=\psi_+\big|_{\p\om}=\psi_-\big|_{\p\om}\in C^\infty(\p\om),
\quad
\Psi^{(1)}=\frac{\p\psi_+}{\p\tau}\Big|_{\p\om}=-\frac{\p\psi_-}{\p\tau}\big|_{\p\om}\in C^\infty(\p\om)
\end{align*}
and
\[
x'=P+\tau \nu(P)
\]
for small positive $\tau$.
\end{lemma}
\begin{proof}
Clearly if $\lambda$ is an eigenvalue of $\Hd$ with eigenfunction $u$,
then $\lambda$ is an eigenvalue of $\Ho$ with eigenfunction $(u,-u)$.
Similarly, an eigenvalue of $\Hn$ with eigenfunction $v$ will also be
an eigenvalue of $\Ho$ with eigenfunction $(v,v)$.

Assume now that $(u,v)$ is an eigenfunction of $\Ho$ and consider the
functions $w_{1}=u-v$ and $w_{2}=u+v$. Then, provided
they do not vanish identically, both $w_{1}$ and $w_{2}$ will be eigenfunctions
of $\Hd$ and $\Hn$, respectively. In case $w_{1}$ vanishes identically, then
$u=v$ and $u$ will be an eigenfuntion of $\Hn$, while if $w_{2}$ vanishes
$u=-v$ and this will be an eigenfunction of $\Hd$.

The remaining part of the lemma follows from standard arguments.
\end{proof}

%D% Let us calculate the metrics on $\hs_\e$.

By $L_2(\bs{\om}, J_\e \di x')$ we indicate the subspace of
$L_2(\bs{\om})$ consisting of the functions $\bs{u}$ with the finite
norm
\begin{align*}
&\|\bs{u}\|_{L_2(\bs{\om}, J_\e\di x')}^2=\|u_+\|_{L_2(\om_+,
J_\e^+\di x')}^2+ \|u_-\|_{L_2(\om_-, J_\e^-\di x')}^2, \quad
\\
&\|u_\pm\|_{L_2(\om,J_\e^\pm\di x')}^2=\int\limits_{\om_\pm}
|u_\pm(x')|^2 J_\e^\pm(x') \di x'.
\end{align*}
In the same way we introduce the space $\H^1(\bs{\om},J_\e \di x')$
as consisting of $\bs{u}\in\H^1(\bs{\om})$ with the finite norm
\begin{equation*}
\|\bs{u}\|_{\H^1(\bs{\om},J_\e\di x')}^2= \|\nabla_{x'}
\bs{u}\|_{L_2(\bs{\om},J_\e\di x')}^2+ \|
\bs{u}\|_{L_2(\bs{\om},J_\e\di x')}^2,
\end{equation*}
where $\nabla_{x'}\bs{u}=(\nabla_{x'} u_+,\nabla_{x'} u_-)$.

\begin{lemma}\label{lm3.1}
The spaces $L_2(\hs_\e)$ and $L_2(\bs{\om}, J_\e \di x')$ are
isomorphic and the isomorphism is the operator $\PS_\e:
L_2(\bs{\om}, J_\e\di x')\to L_2(\hs_\e)$. If $\bs{u}\in
\H^1(\bs{\om},J_\e\di x')$, then $\PS_\e \bs{u}\in\H^1(\hs_\e)$, and
if $u\in\H^1(\hs_\e)$, then $\PS_\e^{-1}u\in\H^1(\bs{\om},J_\e\di
x')$. The inequality
\begin{equation}\label{3.7}
\|J_\e^{-\frac{1}{2}}\nabla_{x'}\bs{u}\|_{L_2(\bs{\om})}\leqslant
\|\nabla \PS_\e \bs{u}\|_{L_2(\hs_\e)}\leqslant \|\nabla_{x'}
\bs{u}\|_{L_2(\bs{\om},J_\e\di x')}
\end{equation}
holds true, where $J_\e^{-\frac{1}{2}}\nabla_{x'}\bs{u}:=\big(
(J_\e^+)^{-\frac{1}{2}}\nabla_{x'}u_+,
(J_\e^-)^{-\frac{1}{2}}\nabla_{x'}u_-\big)$, $\bs{u}=(u_+,u_-)$.
\end{lemma}

\begin{proof}
The fact that $\PS_\e$ is a bijection between the two spaces follows directly
from its definition.

Regarding the inequalities we have
\[
\begin{array}{lll}
\|J_\e^{-\frac{1}{2}}\nabla_{x'}\bs{u}\|_{L_2(\bs{\om})}^2 & = & \int\limits_{\om_{+}}\left(J_\e^{+}\right)^{-1}
|\nabla_{x'}u_{+}|^2dx'+\int\limits_{\om_{-}}\left(J_\e^{-}\right)^{-1}
|\nabla_{x'}u_{-}|^2dx'\eqskip
& = & \int\limits_{\om_{+}} J_\e^{+}\left(J_\e^{+}\right)^{-2} |\nabla_{x'}u_{+}|^2dx'+\int\limits_{\om_{-}}J_\e^{-}\left(J_\e^{-}\right)^{-2}
|\nabla_{x'}u_{-}|^2dx'\eqskip
& \leq & \int\limits_{\om_{+}} J_\e^{+} (\nabla_{x'}u_{+})^{*}G_{+}^{-1}\nabla_{x'}u_{+} dx'+\int\limits_{\om_{-}}
J_\e^{-}(\nabla_{x'}u_{-})^{*}G_{-}^{-1}\nabla_{x'}u_{-} dx'\eqskip
& = & \|\nabla \PS_\e \bs{u}\|_{L_2(\hs_\e)}\eqskip
& \leq & \int\limits_{\om_{+}} J_\e^{+} |\nabla_{x'}u_{+}|^2 dx'+\int\limits_{\om_{-}}J_\e^{-}
|\nabla_{x'}u_{-}|^{2} dx'\eqskip
& = & \|\nabla_{x'} \bs{u}\|_{L_2(\bs{\om},J_\e\di x')},
\end{array}
\]
where we have used the knowledge of the eigenvalues of $G_{\pm}$ and the fact that
$ 1\leq J_\e^{\pm}$.
\end{proof}

Denote
$\om_\d:=\om\cap\{x': 0<\tau<\d\}$. We recall that the set  $\om^\d$
was introduced in (\ref{3.12a}), and in what follows $\bs{\om}^\d$ is
$\om^\d$ considered as a two-sided domain.

\begin{lemma}\label{lm3.3}
If $\bs{u}\in \H^1(\bs{\om})$, respectively, $\bs{u}\in
\H^2(\bs{\om})$, then $\bs{u}\in L_2(\bs{\om},J_\e\di x')$,
respectively, $\bs{u}\in\H^1(\bs{\om},J_\e\di x')$. The inequalities
\begin{align}
&\|\bs{u}\|_{L_2(\bs{\om},J_\e\di x')}\leqslant
C\|\bs{u}\|_{\H^1(\bs{\om})}, \label{3.23a}
\\
&\|\bs{u}\|_{L_2(\bs{\om}^{\e^{4/3}},J_\e\di x')}\leqslant C
\e^{2/3}\|\bs{u}\|_{\H^1(\bs{\om})}, \label{3.23c}
\\
&\|\bs{u}\|_{L_2(\bs{\om}^{\e^{4/3}})}\leqslant
C\e^{2/3}\|\PS_\e\bs{u}\|_{\H^1(\hs_\e)},\label{3.23e}
\\
&\|\bs{u}\|_{\H^1(\bs{\om},J_\e\di x')}\leqslant
C\|\bs{u}\|_{\H^2(\bs{\om})}, \nonumber
\\
& \|\bs{u}\|_{\H^1(\om^{\e^{4/3}},J_\e\di x')}\leqslant C
\e^{2/3}\|\bs{u}\|_{\H^2(\bs{\om})} \label{3.23d}
\end{align}
hold true, where $C$ are positive constants independent of $\e$ and
$\bs{u}$.
\end{lemma}

\begin{proof}
Let $\bs{u}\in\H^1(\bs{\om})$, then
$u_\pm\in\H^1(\bs{\om})$, and  for almost all $P\in\p\om$ the
function $u_\pm\big(P+\,\cdot\,\nu(P)\big)$ belongs to
$\H^1(0,\tau_0)$. Let $\chi=\chi(\tau)$ be an infinitely
differentiable cut-off function vanishing as $\tau\geqslant \tau_0$
and being one as $\tau\leqslant \tau_0/2$. Then $u_\pm=u_\pm\chi$
for $\tau\in[0,\tau_0/2]$, and
\begin{equation*}
u_\pm=\int\limits_{\tau_0}^{\tau} \frac{\p
(u_\pm\chi)}{\p\tau}\di\tau,\quad
|u_\pm\big(P+\tau\nu(P)\big)|^2\leqslant
C\|u_\pm\big(P+\,\cdot\,\nu(P)\big)\|_{\H^1(0,\tau_0)}^2,\quad\tau\in[0,\tau_0/2],
\end{equation*}
where $C$ is a positive constant independent of $P$ and $u_\pm$. We
multiply the last inequality by $J_\e^\pm$, integrate over
$\p\om$, and take into account (\ref{3.9}) to obtain
\begin{equation*}
\int\limits_{\p\om}
\big|u_\pm\big(P+\tau\nu(P)\big)\big|^2|\Det^{-1}
 \mathrm{M}|\di\om\leqslant C\|u_\pm\|_{\H^1(\om^{\tau_0})}^2,
\end{equation*}
where $C$ is a positive constant independent of $P\in\p\om$, and
$u_\pm$. The above estimate, inequality~(\ref{3.12}), the definition
(\ref{3.4}) of $J_\e^\pm$ and the smoothness of $h_\pm$ imply %\corr\corr
\begin{equation}\label{3.14}
\begin{aligned}
&\int\limits_{\om} |u_\pm|^2 J_\e^\pm\di x'= \int\limits_{\om_\d}
|u_\pm|^2 J_\e^\pm\di x'+ \int\limits_{\om^\d} |u_\pm|^2 J_\e^\pm\di
x',\quad \d\in(0,\tau_0/2],
\\
&\int\limits_{\om^\d} |u_\pm|^2 J_\e^\pm\di x'\leqslant C(\d)
\|u_\pm\|_{L_2(\om^\d)}^2, %D%
\\
&\int\limits_{\om_\d} |u_\pm|^2 J_\e^\pm\di x'
=\int\limits_{0}^{\d}\di \tau \int\limits_{\p\om} |u_\pm|^2 J_\e^\pm
|\Det^{-1}\mathrm{M}|\di\om
\\
&\hphantom{\int\limits_{\om^\d} |u_\pm|^2 J_\e^\pm\di x' = }
\leqslant C\|u_\pm\|_{\H^1(\om)}^2
\int\limits_{0}^{\d}\sqrt{1+C_4\e^2\tau^{-1}}\di\tau,
\end{aligned}
\end{equation}
where the constants $C$ and $C(\d)$ are independent of $\e$ and
$u_\pm$, and $C$ is independent of $\d$. Taking $\d=\tau_0/2$, we
see that $\bs{u}\in L_2(\bs{\om},J_\e\di x')$
and thus the estimate (\ref{3.23a}) holds.
If we now take $\d=\e^{4/3}$ in (\ref{3.14}) instead and use the identity
\begin{equation*}
\int\limits_{0}^{\d}\sqrt{1+\e^2C_4\tau^{-1}}\di
\tau=\mathring{J}_\e^\pm(\d):=\sqrt{\d^2+C_4\e^2\d}+
\frac{C_4}{2}\e^2 \ln
\frac{C_4\e^2+2\d+2\sqrt{\d^2+C_4\e^2\d}}{C_4\e^2},
\end{equation*}
we obtain (\ref{3.23c}).

Let us prove (\ref{3.23e}). We integrate by parts as follows, %\corr\corr\corr
\begin{align*}
&\int\limits_{\om_{\e^{4/3}}%D%
} |u_\pm|^2 J_\e^\pm\di x'\leqslant
C\int\limits_{\p\om} \di \om\int\limits_0^{\e^{4/3}} |u_\pm|^2
J_\e^\pm\di\tau,
\\
&
\begin{aligned}
\int\limits_0^{\e^{4/3}} |u_\pm|^2 J_\e^\pm\di\tau =& |u_\pm|^2
\mathring{J}_\e^\pm\Big|_{\tau=0}^{\tau=\e^{4/3}}-
2\int\limits_{0}^{\e^{4/3}} \mathring{J}_\e^\pm(\tau) \RE u_\pm
\frac{\p \overline{u}_\pm}{\p\tau}\di \tau
\\
\leqslant & \mathring{J}_\e^\pm(\e^{4/3}) \bigg(
|u_\pm|^2|_{\tau=\e^{4/3}}+\int\limits_{0}^{\e^{4/3}} |u_\pm|^2
J_\e^\pm \di\tau+ \int\limits_{0}^{\e^{4/3}} \frac{1}{J_\e^\pm}
\Big|\frac{\p u_\pm}{\p\tau}\Big|^2\di \tau\bigg),
\end{aligned}
\\
&
\begin{aligned}
\int\limits_{\om_{\e^{4/3}}%D%
} |u_\pm|^2 J_\e^\pm\di x'\leqslant
C\e^{4/3} \bigg(& \int\limits_{\p\om}
|u_\pm|^2\big|_{\tau=\e^{4/3}}\di\om
\\
&+ \int\limits_{\om_{\e^{4/3}}} \left(\frac{1}{J_\e^\pm}
|\nabla_{x'} u_\pm|^2+ J_\e^\pm |u_\pm|^2\right)\di x'\bigg).
\end{aligned}
\end{align*}
By the embedding of $\H^1(\om^{\e^{4/3}})$ into $L_2(\{x:
\tau=\e^{4/3}\})$ we have the estimate
\begin{equation*}
\int\limits_{\p\om} |u_\pm|^2\big|_{\tau=\e^{4/3}}\di \om\leqslant
C\|u_\pm\|_{\H^1(\om^{\e^{4/3}})}^2\leqslant C\|\PS_\e
\bs{u}\|_{\H^1(\hs_\e)}^2,
\end{equation*}
where the constants $C$ are independent of $\e$ and $\bs{u}$. These
two last estimates together with (\ref{3.7}) yield (\ref{3.23e}).

To prove the second part of the lemma related to the case
$\bs{u}\in\H^2(\om)$ it is sufficient to note that
since $u_\pm, \nabla_{x'} u_\pm\in\H^1(\om)$, by the first part of
the lemma these functions belong to $L_2(\om,J_\e^\pm\di x')$, and
the estimates (\ref{3.23a}), (\ref{3.23c}) are valid for $\bs{u}$
replaced by $\nabla_{x'} \bs{u}$. This completes the proof.
\end{proof}

\begin{proof}[Proof of Theorem~\ref{th1.1}]
Let $f\in L_2(\hs_\e)$, then $\bs{f}:=\PS_\e f\in L_2(\bs{\om},J_\e\di
x')\subset L_2(\bs{\om})$. Denote
$u^{(\e)}:=(\mathcal{H}_\e-z)^{-1}f$,
$\bs{u}^{(0)}:=(\mathcal{H}_0-z)^{-1}\PS_\e^{-1}f$. By the
definition of $\mathcal{H}_\e$ and $\mathcal{H}_0$ we have
\begin{align}
&\he[u^{(\e)},\vp]-z(u^{(\e)},\vp)_{L_2(\hs_\e)}=(f,\vp)_{L_2(\hs_\e)}
&& \text{for each}\quad\vp\in \H^1(\hs_\e), \label{3.16}
\\
&\ho[\bs{u}^{(0)},\bs{\vp}]-z(\bs{u}^{(0)},\bs{\vp})_{L_2(\bs{\om
})}=(\bs{f},\bs{\vp})_{L_2(\bs{\om})} && \text{for each}\quad\vp\in
\H^1(\bs{\om}).\label{3.17}
\end{align}
Since $\bs{u}^{(0)}\in\H^2(\bs{\om})$, by
Lemmas~\ref{lm3.2} and~\ref{lm3.3}
$u^{(0)}:=\PS_\e\bs{u}^{(0)}\in\H^1(\hs_\e)$. Hence,
$v^{(\e)}:=u^{(\e)}-u^{(0)}\in\H^1(S_\e)$ and this can be
used as a test function in (\ref{3.16}),
\begin{equation*}
\he[u^{(\e)},v^{(\e)}]-z(u^{(\e)},v^{(\e)})_{L_2(S_\e)}=(f,v^{(\e)})_{L_2(S_\e)}.
\end{equation*}
The identity $u^{(\e)}=v^{(\e)}+u^{(0)}$ yields
\begin{equation}\label{3.19}
\begin{aligned}
\|\nabla v^{(\e)}\|_{L_2(S_\e)}^2-&z\|v^{(\e)}\|_{L_2(S_\e)}^2
\\
&=(f,v_\e)_{L_2(S_\e)}-(\nabla u^{(0)},\nabla v^{(\e)})_{L_2(S_\e)}+
z (u^{(0)},v^{(\e)})_{L_2(S_\e)}.
\end{aligned}
\end{equation}
We parameterize $S_\e$ as $x'=x'$, $x_{n+1}=\pm \e h_\pm(x')$, and
use the definition of the scalar product $(\nabla u^{(0)},\nabla
v^{(\e)})_{L_2(S_\e)}$. It implies
\begin{align*}
(f,&v^{(\e)})_{L_2(S_\e)}-(\nabla u^{(0)},\nabla
v^{(\e)})_{L_2(S_\e)}+ z(u^{(0)},v^{(\e)})_{L_2(S_\e)}
\\
= & \hspace*{2.5mm} (f_+,J_\e^+ v^{(\e)}_+)_{L_2(\om_+)}+ (f_-,J_\e^-
v^{(\e)}_-)_{L_2(\om_-)}
\\
& \hspace*{5mm} - \left( (J_\e^+ G_+^{-1}\nabla_{x'} u^{(0)}_+, \nabla_{x'}
v^{(\e)}_+)_{L_2(\om_+)}+(J_\e^- G_-^{-1}\nabla_{x'} u^{(0)}_-, \nabla_{x'} v^{(\e)}_-)_{L_2(\om_-)} \right)
\\
& \hspace*{7.5mm} +z(u^{(0)}_+, J_\e^+ v^{(\e)}_+)_{L_2(\om_+)}+z (u^{(0)}_-,
J_\e^- v^{(\e)}_-)_{L_2(\om_-)},
\end{align*}
where
$\bs{v}^{(\e)}=(v^{(\e)}_+,v^{(\e)}_-)=\PS_\e^{-1}v^{(\e)}$ and
$G^{ij}_\pm$ are the entries of the inverse matrix $G_\pm^{-1}$. We
substitute the last formula into (\ref{3.19}) and then sum it with
(\ref{3.17}), where we take
$\bs{\vp}=\bs{v}^{(\e)}\in\H^1(\bs{\om},J_\e\di x')\subset
\H^1(\bs{\om})$,
\begin{gather}
\|\nabla v^{(\e)}\|_{L_2(S_\e)}^2-z\|v^{(\e)}\|_{L_2(S_\e)}^2=R^+
+R^-,\label{3.20}
\\
\begin{aligned}
R^\pm:=&(f_\pm,(J_\e^\pm-1)v^{(\e)}_\pm)_{L_2(\om)}-
 (J_\e^\pm G_\pm^{-1}\nabla_{x'} u^{(0)}_\pm, \nabla_{x'}
v^{(\e)}_\pm)_{L_2(\om_\pm)}
\\
&-(\nabla_{x'} u^{(0)}_\pm,\nabla_{x'} v^{(\e)}_\pm)_{L_2(\om)} +z
(u^{(0)}_\pm,(J_\e^\pm-1)v^{(\e)}_\pm)_{L_2(\om)}.
\end{aligned}\nonumber
\end{gather}
Let us estimate $R^\pm$ which we shall write as
\begin{gather}
R^\pm=R_1^\pm + R_2^\pm,\label{3.21}
\mbox{ where }
\\
\begin{aligned}
R_1^\pm:=&(f_\pm,(J_\e^\pm-1)v^{(\e)}_\pm)_{L_2(\om^{\d})}-
 (J_\e^\pm G_\pm^{-1}\nabla_{x'} u^{(0)}_\pm, \nabla_{x'}
v^{(\e)}_\pm)_{L_2(\om^{\d})}
\\
&-(\nabla_{x'} u^{(0)}_\pm,\nabla_{x'} v^{(\e)}_\pm)_{L_2(\om^{\d})}
+z (u^{(0)}_\pm,(J_\e^\pm-1)v^{(\e)})_{L_2(\om^{\d})}.
\end{aligned}\nonumber
\\
\begin{aligned}
R_2^\pm:=&(f_\pm,(J_\e^\pm-1)v^{(\e)}_\pm)_{L_2(\om^{\d})}-
 (J_\e^\pm G_\pm^{-1}\nabla_{x'} u^{(0)}_\pm, \nabla_{x'}
v^{(\e)}_\pm)_{L_2(\om^{\d})}
\\
&-(\nabla_{x'} u^{(0)}_\pm,\nabla_{x'} v^{(\e)}_\pm)_{L_2(\om^{\d})}
+z (u^{(0)}_\pm,(J_\e^\pm-1)v^{(\e)}_\pm)_{L_2(\om^{\d})},
\end{aligned}\nonumber
\end{gather}
and $\d:=\e^{4/3}$. As $x'\in\om_\d$, by (\ref{3.12}) we have
\begin{align*}
&\e^2|\nabla_{x'} h_\pm|^2\leqslant C\e^{2/3}, &&
\|\mathrm{G}_\pm^{-1}-\mathrm{E}\|\leqslant C\e^{2/3},
\\
&|J_\e^\pm-1|\leqslant C\e^{2/3}, && |(J_\e^\pm)^{-1}-1|\leqslant
C\e^{2/3}.
\end{align*}
Hereinafter by $C$ we indicate non-essential positive constants
independent of $\e$, $u^{(\e)}$, $\bs{u}^{(0)}$, and $f$. Hence, by
Lemmas~\ref{lm3.2},~\ref{lm3.3} and Schwarz's inequality
\begin{align*}
&\big|(f_\pm,(J_\e^\pm-1)v_\pm^{(\e)})_{L_2(\om_\d)}\big|\leqslant
C\e^{2/3}\|f_\pm\|_{L_2(\om,J_\e^\pm\di
x')}\|v^{(\e)}_\pm\|_{L_2(\om,J_\e^\pm\di x')}
\\
&\hphantom{\big|(f_\pm,(J_\e^\pm-1)v_\pm^{(\e)})_{L_2(\om_\d)}\big|}
\leqslant C\e^{2/3} \|f\|_{L_2(S_\e)}\|v^{(\e)}\|_{L_2(S_\e)},
\\
&\big|z(u^{(0)}_\pm,(J_\e^\pm-1) v_\pm^{(\e)})_{L_2(\om_\d)}\big|
\leqslant C\e^{2/3}
\|u^{(0)}\|_{L_2(\om)}\|v^{(\e)}\|_{L_2(S_\e)},
\\
&\Big|
(\nabla_{x'}u_\pm^{(0)},\nabla_{x'}v_\pm^{(\e)})_{L_2(\om_\d)}-
 (J_\e^\pm G_\pm^{-1}\nabla_{x'} u^{(0)}_\pm, \nabla_{x'}
v^{(\e)}_\pm)_{L_2(\om_\d)}\Big|
\\
&\hphantom{ \Big| (\nabla_{x'}u_\pm^{(0)},\nabla_{x'}v_\pm^{(\e)})
_{L_2(\om_\d)}-} \leqslant
C\e^{2/3}\|\bs{u}^{(0)}\|_{\H^1(\bs{\om})}\|\nabla_{x'}
v_\pm^{(\e)}\|_{L_2(\om_\d)}
\\
&\hphantom{ \Big| (\nabla_{x'}u_\pm^{(0)},\nabla_{x'}v_\pm^{(\e)})
_{L_2(\om_\d)}-} \leqslant
C\e^{2/3}\|\bs{u}^{(0)}\|_{\H^1(\bs{\om})}\|J_\e^{-\frac{1}{2}}\nabla_{x'}
\bs{v}^{(\e)}\|_{L_2(\bs{\om}_\d)}
\\
&\hphantom{ \Big| (\nabla_{x'}u_\pm^{(0)},\nabla_{x'}v_\pm^{(\e)})
_{L_2(\om_\d)}-} \leqslant
C\e^{2/3}\|\bs{u}^{(0)}\|_{\H^1(\bs{\om})}\|J_\e^{-1}\nabla_{x'}
\bs{v}^{(\e)}\|_{L_2(\bs{\om}_\d,J_\e\di x')}
\\
&\hphantom{ \Big| (\nabla_{x'}u_\pm^{(0)},\nabla_{x'}v_\pm^{(\e)})
_{L_2(\om_\d)}-} \leqslant
C\e^{2/3}\|\bs{u}^{(0)}\|_{\H^1(\bs{\om})}\|\nabla
v^{(\e)}\|_{L_2(S_\e)},
\end{align*}
and therefore
\begin{equation}\label{3.22}
|R_1^+ + R_1^-|\leqslant C\e^{2/3} \|\bs{u}^{(0)}\|_{\H^1(\bs{\om})}
\|v^{(\e)}\|_{\H^1(S_\e)}.
\end{equation}
To estimate $R_2^\pm$ we employ (\ref{3.23a}), (\ref{3.23c}),
(\ref{3.23e}). We begin with the first term in $R_2^\pm$ applying again
Schwarz's inequality and~(\ref{3.23e}) to obtain
\begin{equation}\label{3.25}
\begin{aligned}
|(f_\pm,(J_\e^\pm-1)v^{(\e)}_\pm)_{L_2(\om^{\d})}|\leqslant &
\|f_\pm\|_{L_2(\om^\d,J_\e^\pm\di x')}
\|\big(1-(J_\e^\pm)^{-1}\big)v^{(\e)}_\pm\|_{L_2(\om^\d, J_\e^\pm\di
x')}
\\
\leqslant & \|f\|_{L_2(S_\e)} \|v^{(\e)}_\pm\|_{L_2(\om^\d,
J_\e^\pm\di x')}
\\
\leqslant & C\e^{2/3} \|f\|_{L_2(S_\e)} \|v^{(\e)}\|_{\H^1(S_\e)}.
\end{aligned}
\end{equation}
Employing (\ref{3.7}), (\ref{3.23a}) and (\ref{3.23e}) in the same way we get two more estimates,
\begin{equation}\label{3.26}
\begin{aligned}
&
\begin{aligned}
|z(u^{(0)}_\pm,(J_\e^\pm-1) v^{(\e))})_{L_2(\om^\d)}|\leqslant & C
\|u^{(0)}_\pm\|_{L_2(\om^\d,J_\e^\pm\di x')}
\|v^{(\e)}_\pm\|_{L_2(\om^\d,J_\e^\pm\di x')}
\\
\leqslant & C\e^{2/3} \|\bs{u}^{(0)}\|_{\H^1(\bs{\om})}
\|v^{(\e)}\|_{\H^1(S_\e)},
\end{aligned}
\\
&
\begin{aligned}
|(\nabla_{x'}u_\pm^{(0)},\nabla_{x'}
v_\pm^{(\e)})_{L_2(\om^\d)}|\leqslant &
\|(J_\e^\pm)^{\frac{1}{2}}\nabla_{x'} u_\pm^{(0)}\|_{L_2(\om^\d)}
\|(J_\e^\pm)^{-\frac{1}{2}}\nabla_{x'} v_\pm^{(\e)}\|_{L_2(\om^\d)}
\\
\leqslant & C\e^{2/3} \|\bs{u}^{(0)}\|_{\H^2(\bs{\om})} \|\nabla
v^{(\e)}\|_{L_2(S_\e)}.
\end{aligned}
\end{aligned}
\end{equation}
Since
\begin{equation*}
(G_\pm^{-1} \nabla_{x'} u^{(0)}_\pm,\nabla_{x'}
v^{(\e)}_\pm)_{\mathds{R}^n}=\nabla \PS_\e\bs{u}^{(0)}\cdot \nabla v^{(\e)},
\end{equation*}
by Schwarz's inequality we have
\begin{align*}
\Big|(G_\pm^{-1} \nabla_{x'} u^{(0)}_\pm,\nabla_{x'}
v^{(\e)}_\pm)_{L_2(\om^{\d})}\Big|\leqslant & \|\nabla
v^{(\e)}\|_{L_2(S_\e)} (G_\pm^{-1} \nabla_{x'} u^{(0)}_\pm,\nabla_{x'}
u^{(0)}_\pm)_{L_2(\om^{\d})}^{\frac{1}{2}}
\\
\leqslant & \|\nabla v^{(\e)}\|_{L_2(S_\e)}
\|(J_\e^\pm)^{\frac{1}{2}}\nabla_{x'} u_\pm^{(0)}\|_{L_2(\om^\d)}.
\end{align*}
Here we have used the inequality
\begin{equation*}
\sum\limits_{i,j=1}^n G^{ij}_\pm \xi_i \xi_j\leqslant
\sum\limits_{i=1}^n|\xi_i|^2,
\end{equation*}
which follows from Lemma~\ref{lm3.2}. Using~(\ref{3.23d}) we get
\begin{align*}
\Big|(G_\pm^{-1} \nabla_{x'} u^{(0)}_\pm,\nabla_{x'}
v^{(\e)}_\pm)_{L_2(\om^{\d})}\Big|\leqslant & \|\nabla
v^{(\e)}\|_{L_2(S_\e)} \|\bs{u}^{(0)}\|_{\H^1(\bs{\om}^\d)}
\\
\leqslant & C\e^{2/3} \|\nabla v^{(\e)}\|_{L_2(S_\e)}
\|\bs{u}^{(0)}\|_{\H^2(\bs{\om})},
\end{align*}
which with~(\ref{3.25}) and (\ref{3.26}) yield
\begin{equation*}
|R_2^+ + R_2^-|\leqslant C\e^{2/3} \|\bs{u}^{(0)}\|_{\H^2(\bs{\om})}
\|v^{(\e)}\|_{\H^1(S_\e)}.
\end{equation*}
Together with (\ref{2.7a}), (\ref{3.20}), (\ref{3.21}), (\ref{3.22})
it follows that %\corr\corr
\begin{align*}
\big|\|\nabla v^{(\e)}\|_{L_2(S_\e)}^2-z\|v^{(\e)}\|_{L_2(S_\e)}^2
\big| &\leqslant C \e^{2/3} \|\bs{u}^{(0)}\|_{\H^2%D%
(\bs{\om})}
\|v^{(\e)}\|_{\H^1(S_\e)}
\\
&\leqslant C\e^{2/3}
\|\bs{f}\|_{L_2(\bs{\om})}\|v^{(\e)}\|_{\H^1(S_\e)}.
\end{align*}
Since
\begin{equation*}
\big|\|\nabla v^{(\e)}\|_{L_2(S_\e)}^2-z\|v^{(\e)}\|_{L_2(S_\e)}^2
\big| \geqslant C\|v^{(\e)}\|_{\H^1(S_\e)}^2,
\end{equation*}
we arrive at (\ref{1.8}), completing the proof.
\end{proof}

\begin{remark}
 The proof above uses the estimates from Lemma~\ref{lm3.3} which include a measure of the boundary behaviour
 by means of the weight function $J_{\e}$. A different approach which may
 also be used to prove convergence of the resolvent in similar situations is based on inequalities of
 Hardy type instead, possibly allowing for a better control of the behaviour near the boundary -- see~\cite{krzu}
 for an illustration of this principle.
\end{remark}

In the proof of Theorem~\ref{th2.3} in the next section we shall use
the following auxiliary lemma which is convenient to prove in this section.

\begin{lemma}\label{lm3.4}
Let $\l$ be a $m$-multiple eigenvalue of $\Ho$, and $\l_i(\e)$,
$i=1,\ldots,m$, be the eigenvalues of $\He$ taken counting
multiplicity and converging to $\l$, and $\psi_\e^{(i)}$ be the
associated eigenfunctions orthonormalized in $L_2(S_\e)$. For $z$
close to $\l$ the representation
\begin{equation*}
(\He-z)^{-1}=\sum\limits_{i=1}^{m}
\frac{\psi_\e^{(i)}}{\l_i(\e)-z}(\cdot,
\psi_\e^{(i)})_{L_2(S_\e)} + \mathcal{R}_\e(z)
\end{equation*}
holds true, where the operator $\mathcal{R}_\e(z):
L_2(S_\e)\to\H^1(S_\e)$ is bounded uniformly in $\e$ and $z$. The
range of $\mathcal{R}_\e(z)$ is orthogonal to all $\psi_\e^{(i)}$,
$i=1,\ldots,m$.
\end{lemma}

\begin{proof}
We choose a fixed $\d$ so that the disk $B_\d(\l):=\{z:
|z-\l|<\d\}$ contains no eigenvalues of $\mathcal{H}_0$ except
$\l$ and
\begin{equation*}
\dist\{\p B_\d(\l), \discspec(\Ho)\}\geqslant \d.
\end{equation*}
Then, by Theorem~\ref{th2.2}, for sufficiently small $\e$ this disk
contains the eigenvalues $\l_i(\e)$, $i=1,\ldots,m$, and no other
eigenvalues of $\He$, and
\begin{equation}\label{3.31}
\dist\big\{B_\d(\l), \discspec(\He)\setminus
\{\l_i(\e),i=1,\ldots,m\} \big\}\geqslant \frac{\d}{2}.
\end{equation}
Denote by $V_\e$ the orthogonal complement to $\psi_\e^{(i)}$,
$i=1,\ldots,m$, in $L_2(S_\e)$. By \cite[Ch. V, Sec. 3.5, Eqs.
(3.21)]{K} the representation (3.29) holds true, where
$\mathcal{R}_\e(z)$ is the part of the resolvent $(\He-z)^{-1}$
acting in $V_\e$ and
\begin{equation}\label{3.32}
\|\mathcal{R}_\e(z)\|_{V_\e\to V_\e}\leqslant
\frac{1}{\dist\big\{B_\d(\l), \discspec(\He)\setminus
\{\l_i(\e),i=1,\ldots,m\} \big\}}\leqslant \frac{2}{\d}
\end{equation}
for $z\in B_\d(\l)$, where we have used (\ref{3.31}). Hence, the
range of $\mathcal{R}_\e(z)$ is orthogonal to $\psi_\e^{(i)}$,
$i=1,\ldots,m$. It is easy to check that the function
$u_\e:=\mathcal{R}_\e(z) f$, $f\in L_2(S_\e)$ solves the equation
\begin{equation*}
(\He-z)u_\e=f_\e,\quad f_\e:=f-\sum\limits_{i=1}^{m}\psi_\e^{(i)}
(f,\psi_\e^{(i)})_{L_2(S_\e)},\quad \|f_\e\|_{L_2(S_\e)}\leqslant
\|f\|_{L_2(S_\e)}.
\end{equation*}
Hence, by the definition of $\He$ and (\ref{3.32})
\begin{align*}
\|\nabla
u_\e\|_{L_2(S_\e)}^2=&z\|u_\e\|_{L_2(S_\e)}^2+(f_\e,u_\e)_{L_2(S_\e)}
\leqslant
|z|\|u_\e\|_{L_2(S_\e)}^2+\|f_\e\|_{L_2(S_\e)}\|u_\e\|_{L_2(S_\e)}
\\
&\leqslant C(\d)\|f\|_{L_2(S_\e)}^2,
\end{align*}
where the constant $C(\d)$ is independent of $\e$ and $f$. The last estimate and (\ref{3.32}) complete the proof.
\end{proof}

\section{Asymptotic expansions\label{sec4}}

In this section we give the proof of Theorem~\ref{th2.3} which will be divided
into two parts. We first build the asymptotic expansions
formally, where the core of the formal construction is the
method of matching asymptotic expansions \cite{Il}. The second part is devoted to the justification of the asymptotics, i.e., obtaining estimates for the error terms.

The formal construction consists of determining the outer and inner expansions on the base of the perturbed eigenvalue problem and the matching of
these expansions. The outer expansion is used to approximate the perturbed eigenfunctions outside a small neighborhood of $\p\om$. It is
constructed in terms of the variables $x'$ using the first parametrization of $\hs_\e$ given in the previous sections. In a vicinity of $\p\om$ the
perturbed eigenfunctions are approximated by the inner expansion which is based on the second parametrization of $\hs_\e$ and is constructed in terms
of the variables $(\xi,s)$.

\subsection{Outer expansion: first term}

By Theorem~\ref{th2.2} there exist exactly $m$ eigenvalues of $\He$
converging to $\l$ counting multiplicities. We denote these
eigenvalues by $\l_k(\e)$, $k=1,\ldots,m$, while the symbols
$\psi_\e^{(k)}$ will denote the associated eigenfunctions. We
construct the asymptotics for $\l_k(\e)$ as %\corr\corr
\begin{equation}\label{4.1}
\l_k(\e)=\l+\e^2\ln\e\,
\mu_k%D%
\left(\frac{1}{\ln\e}\right) %D%
+\ldots
\end{equation}
Hereinafter terms like $\ln\e A$ are understood as $(\ln\e) A$. In
accordance with the method of matching asymptotic expansions we
form the asymptotics for $\psi_\e^{(k)}$ as the sum of outer
and inner expansions. The outer expansion is built as
\begin{equation}\label{4.2}
\psi_{\e,ex}^{(k)}=\PS_\e
(\bs{\psi}_k+ \e^2\ln\e\,
\bs{\phi}_k+\ldots),
\end{equation}
where $\bs{\phi}_k=(\phi_+^{(k)},\phi_-^{(k)})$,
$\phi_\pm^{(k)}=\phi_\pm^{(k)}(x',\e)$, and the eigenfunctions $\bs{\psi}_k$
are chosen as described before the statement of Theorem~\ref{th2.3} in Sec.~2. We also recall
that these functions depend on $\e$ in the case where $\l$ is a multiple eigenvalue.

We substitute the  identities~(\ref{4.1}),~(\ref{4.2}),~and~(\ref{5.1})
into the eigenvalue equation
\begin{equation}\label{4.5}
\He\psi_\e^{(k)}=\l_k(\e)\psi_\e^{(k)},
\end{equation}
and take into account the eigenvalue equations for
$\bs{\psi}_i$. It implies the equations for
$\bs{\phi}_k$, namely,
\begin{equation}
\begin{aligned}
&(-\D_{x'}-\l)\phi_\pm^{(k)}=\frac{1}{\ln\e} f_{2,\pm}^{(k)}+
\mu_k\psi_\pm^{(k)},\quad x'\in\om_\pm,
\qquad
f_{2,\pm}^{(k)}:=\mathcal{H}_\pm^{(2)}\psi_\pm^{(k)},
\\
&
\mathcal{H}_\pm^{(2)}:=-\Div_{x'} \mathrm{Q}_\pm\nabla_{x'}
-
\frac{|\nabla_{x'} h_\pm|^2}{2}\D_{x'}
+\frac{1}{2}\Div_{x'}
|\nabla_{x'} h_\pm|^2 \nabla_{x'}.
\end{aligned} \label{4.6b}
\end{equation}

The functions $\psi_\pm^{(i)}$ are infinitely differentiable in
$\overline{\om}_\pm$, and thus
\begin{equation}\label{4.7}
\psi_\pm^{(k)}(x',\e)=\Psi_k^{(0)}(P,\e) \pm \Psi_k^{(1)}(P,\e)\tau +
\Psi_k^{(2,\pm)}(P,\e)\tau^2+\Odr(\tau^3),\quad P\in\p\om,
\end{equation}
as $\tau\to+0$, where by the definition of the domain of $\Ho$
\begin{align*}
&\Psi_k^{(0)}:=\psi_+^{(k)}\big|_{\p\om}=
\psi_-^{(k)}\big|_{\p\om}, &&
\Psi_k^{(1)}:=\frac{\p\psi_+^{(k)}}{\p\tau}\bigg|_{\p\om}=
-\frac{\p\psi_-^{(k)}}{\p\tau}\bigg|_{\p\om},
\\
&\Psi_k^{(2,\pm)}:=\frac{1}{2}
\frac{\p^2\psi_\pm^{(k)}}{\p\tau^2}\bigg|_{\p\om},&&
\Psi_k^{(j)},\Psi_k^{(2,\pm)}\in C^\infty(\p\om).
\end{align*}
The functions $\Psi_k^{(i)}$ depend on $\e$ only if $\l$ is a multiple eigenvalue, since the same is true for the  functions $\bs{\psi}_k$.

In view of the identity (\ref{4.55}) we rewrite (\ref{4.7}) as
\begin{align}
&\psi_\pm^{(k)}(x',\e)= \Psi_k^{(0)}(P,\e)\pm \Psi_k^{(1)}(P,\e)\z^2
 +\Psi_k^{(2,\pm)}(P,\e)\z^4+\Odr(\z^6), && \z\to+0. \nonumber
\\
&\psi_\pm^{(k)}(x',\e)= \Psi_k^{(0)}(P,\e)\pm
\e^2\Psi_k^{(1)}(P,\e)\xi^2
+\e^4\Psi_k^{(2,\pm)}(P,\e)\xi^4+\Odr(\e^6\xi^6), && \e\xi\to0.
\label{5.5a}
\end{align}

\subsection{Inner expansion}

In accordance with the method of matching asymptotic expansions the identities (\ref{4.2}), (\ref{5.5a}) yield that the inner expansion for the eigenfunctions $\psi_\e^{(k)}$ should read  as follows,
\begin{equation}\label{4.12}
\psi_{\e,in}^{(k)}(\xi,P,\e)=\sum\limits_{i=0}^{4} \e^i
v_i^{(k)}(\xi,P,\e)+\ldots,
\end{equation}
where the coefficients must satisfy the following asymptotics as
$\xi\to\pm\infty$
\begin{align}
& v_0^{(k)}(\xi,P,\e)=\Psi_k^{(0)}(P,\e)+o(1),\label{4.13}
\\
&v_1^{(k)}(\xi,P,\e)=o(|\xi|),\label{4.18a}
\\
&v_2^{(k)}(\xi,P,\e)=\pm \Psi_k^{(1)}(P,\e)\xi^2+o(|\xi|^2),\label{4.14}
\\
&v_3^{(k)}(\xi,P,\e)=o(|\xi|^3),\nonumber
\\
&v_4^{(k)}(\xi,P,\e)=\Psi_k^{(2,\pm)}(P,\e)\xi^4+o(|\xi|^4). \nonumber
\end{align}
These asymptotics mean that the first term of the outer expansion is matched with the inner expansion.

We  substitute (\ref{4.1}),  (\ref{4.12}), (\ref{4.32}),
 (\ref{4.41}) into the eigenvalue equation (\ref{4.5})
and equate the coefficients of $\e^{-4}$. This implies the equation for $v_0^{(k)}$,
\begin{equation*}
\mathcal{L}_{-4} v_0^{(k)}\equiv -\frac{1}{\sqrt{4\xi^2+b_1^2}}
\frac{\p}{\p\xi} \frac{1}{\sqrt{4\xi^2+b_1^2}} \frac{\p
v_0^{(k)}}{\p\xi}=0\quad\text{on}\quad \mathds{R}\times \p\om.
\end{equation*}
The solution to the last equation satisfying (\ref{4.13}) is
obviously as follows,
\begin{equation}\label{4.45}
v_0^{(k)}(\xi,P,\e)\equiv \Psi_k^{(0)}(P,\e).
\end{equation}
We then substitute this identity and (\ref{4.1}), (\ref{4.12}), (\ref{4.32}), (\ref{4.39a}),
(\ref{4.40a}), (\ref{4.32}) into (\ref{4.5}) and
equate the coefficients at $\e^i$, $i=-3,\ldots,0$, leading us
to the equations for $v_i^{(k)}$, $i=1,\ldots,4$,
\begin{align}
& \mathcal{L}_{-4} v_1^{(k)}=0\quad \text{on}\quad
\mathds{R}\times\p\om, \label{4.46a}
\\
& \mathcal{L}_{-4} v_2^{(k)}=0\quad \text{on}\quad
\mathds{R}\times\p\om, \label{4.46}
\\
&\mathcal{L}_{-4} v_3^{(k)}+ \mathcal{L}_{-3} v_2^{(k)} +
\mathcal{L}_{-2} v_1^{(k)}=0 \quad \text{on}\quad
\mathds{R}\times\p\om, \label{4.46c}
\\
&\mathcal{L}_{-4} v_4^{(k)}+ \mathcal{L}_{-3} v_3^{(k)} +
\mathcal{L}_{-2} v_2^{(k)} +
 \mathcal{L}_{-1} v_1^{(k)} +  \mathcal{L}_0 v_0^{(k)} =\l v_0^{(k)} \quad
\text{on}\quad \mathds{R}\times\p\om, \label{4.46d}
\end{align}
were we have used that
\begin{equation*}
\mathcal{L}_i v_0^{(k)}\equiv 0,\quad i=-3,\ldots,-1,
\end{equation*}
due to (\ref{4.39a}), (\ref{4.40a}), (\ref{4.45}). The only solution
to (\ref{4.46a}) satisfying (\ref{4.18a}) is  independent of $\xi$,
\begin{equation}\label{4.46b}
v_1^{(k)}(\xi,P,\e)\equiv C_1^{(k,0)}(P,\e),
\end{equation}
where $C_1^{(k,0)}$ is an unknown function to be determined.

The equation (\ref{4.46}) can be solved, and the solution satisfying
(\ref{4.14}) is  
\begin{align}
v_2^{(k)}(\xi,P,\e)=&\Psi_k^{(1)}(P,\e) X_1(\xi, b_1(P))+ C_2^{(k,0)}(P,\e),
\label{4.47}
\\
X_1(\xi,b):=&\frac{1}{2}\xi (4\xi^2+b^2)^{\frac{1}{2}}+\frac{b^2}{4}\ln
\big(2 \xi+(4\xi^2+b^2)^{\frac{1}{2}}\big)-\frac{b^2}{4}\ln b,\label{4.48a}
\end{align}
where $C_2^{(k,0)}$ is an unknown function to be determined.

In view of (\ref{4.46b}), (\ref{4.47}), (\ref{4.39a}),
(\ref{4.40a}) and (\ref{4.46}), equation (\ref{4.46c}) may be written as
\begin{equation*}
\b_{-4} \frac{\p}{\p\xi} \b_{-4} \frac{\p v_3^{(k)}}{\p\xi} =
-\b_{-4} \frac{\p}{\p\xi} \b_{-3} \frac{\p
v_2^{(k)}}{\p\xi}\quad\text{on}\quad \mathds{R}\times\p\om.
\end{equation*}
Employing the formulas (\ref{4.41}), (\ref{4.47}) and (\ref{4.48a}), we
solve the last equation,
\begin{equation}\label{4.49a}
\begin{aligned}
v_3^{(k)}(\xi,P,\e)=&\frac{\Psi^{(k,1)}_0(P,\e)b_1(P)b_2(P)}{2\b_{-4}(\xi,P)}+
C_3^{(k,1)}(P,\e)X_1(\xi)+C_3^{(k,0)}(P,\e)
\\
=&\frac{1}{2} \Psi_k^{(1)}(P,\e)b_1(P) b_2(P) (4\xi^2+b_1^2(P))^{\frac{1}{2}}
\\
&+
C_3^{(k,1)}(P,\e)X_1(\xi)+C_3^{(k,0)}(P,\e),
\end{aligned}
\end{equation}
where $C_3^{(k,1)}$ and $C_3^{(k,0)}$ are unknown functions to be determined.

We substitute (\ref{4.46b}), (\ref{4.47}), (\ref{4.48a}),
(\ref{4.49a}), (\ref{4.39a}), (\ref{4.40a}), (\ref{4.41b}),
(\ref{4.27a}) and (\ref{4.41}) into equation~(\ref{4.46d}) and then
solve it to obtain
\begin{align*}
v_4^{(k)}=&\frac{1}{16}\Psi^{(k,1)}_0\xi \bigg( K
(4\xi^2+b_1^2)^{\frac{3}{2}}
+ 12 b_1b_3(4\xi^2+b_1^2)^{\frac{1}{2}}
+ \frac{8 b_2^2(8\xi^2+3 b_1^2)}{(4\xi^2+b_1^2)^{\frac{1}{2}}}
\bigg)
\\
&+\frac{1}{2} C_3^{(k,1)} b_1 b_2 (4\xi^2+b_1^2)^\frac{1}{2}
-\frac{1}{2}X_1^2 (\D_{\p\om}+\l)\Psi_k^{(0)}
\\
&+ \frac{1}{2} X_2  b_1\nabla b_1\cdot \nabla \Psi_k^{(0)}+ C_4^{(k,1)} X_1  +C_4^{(k,0)},
\end{align*}
where $X_1=X_1(\xi,b_1(P))$,
\begin{align*}
&X_2=X_2(\xi,b):=\xi^2-b^2 X_3\bigg(\frac{2\xi+\sqrt{4\xi^2+b^2} }{b}\bigg),
\\
&X_3(z):=\frac{1}{8}\ln^2z+\frac{1}{16}\left(z^2- \frac{1}{z^2}\right)\ln z-\frac{1}{32} \left(z^2 + \frac{1}{z^2}\right),
\end{align*}
and $C_4^{(k,0)}=C_4^{(k,0)}(P,\e)$ and $C_4^{(k,1)}=C_4^{(k,1)}(P,\e)$ are unknown functions to be determined.

%\corr%D%
To determine the coefficient $\bs\phi^{(k)}$ in the outer expansion and the functions  $C^{k,j}_i$ in the inner one, we should match the constructed
functions $v_i^{(k)}$ with the outer expansion.  In order to do it, we must find the asymptotics for the functions $v_i^{(k)}$ as
$\xi\to\pm\infty$. We observe that the functions $X_1,X_2\in C^\infty(\mathds{R}\times(0,+\infty))$ satisfy
the identities
\begin{align*}
X_1(\xi,b)=&\pm\xi^2\pm\frac{b^2}{8} (2\ln|\xi|+1+4\ln 2-2\ln b) + \Odr(\xi^{-2}), && \xi\to\pm\infty,
\\
X_2(\xi,b)=&\xi^2\left(\frac{3}{2}-2\ln 2+ \ln b-\ln|\xi|\right)+\Odr(\ln^2|\xi|),&& \xi\to\pm\infty,
\end{align*}
uniformly in $b\geqslant b_0>0$, with $b_0$ any fixed constant. Taking
these asymptotics into account, we write the asymptotics for $v_i^{(k)}$ as $\xi\to\pm\infty$ and then pass
to the variables $(\tau,P)$,
\begin{align*}
\sum\limits_{i=0}^{4} &\e^i v_i^{(k)}(\xi,P,\e)=\Psi_k^{(0)}(P,\e)\pm \Psi_k^{(1)}(P,\e)\tau
\\
&+
\frac{1}{2}\left(\pm  \Psi_k^{(1)}(P,\e) K(P)-\D_{\p\om}\Psi_k^{(0)}(P,\e)- \l\Psi_k^{(0)}(P,\e)\right)\tau^2
\\
&+\e(\pm C_3^{(k,1)}(P,\e)\tau+C_1^{(k,0)})
\\
&+\e^2\big(\ln\e\, W_{2,1,\pm}^{(k)}(x',\e)
+ W_{2,0,\pm}^{(k)}(x',\e)\big)
+\Odr(\e^3+\e^4\tau^{-1}),
\end{align*}
where
\begin{align}
&W_{2,1,\pm}^{(k)}:=\frac{1}{4} b_1^2\left( \mp \Psi_k^{(1)}+\tau \left(\D_{\p\om} +\frac{2}{b_1} \nabla b_1\cdot\nabla +\l \right)\Psi_k^{(0)}
\right),\label{4.56a}
\\
&
\begin{aligned}
W_{2,0,\pm}^{(k)}:=&\pm \frac{1}{8}b_1^2\Psi_k^{(1)} \ln \tau
 \pm \frac{b_1^2}{8}(1+4\ln 2-2\ln b_1) \Psi_k^{(1)}+C_2^{(k,0)}
\\
&+\Psi_k^{(1)}b_1 b_2 \tau^{1/2}
 -\frac{1}{8}b_1^2\tau\ln\tau \Big(\D_{\p\om}
+ \frac{2}{b_1}\nabla b_1\cdot\nabla + \l
\Big)\Psi_k^{(0)}
\\
&+\tau \bigg(-\frac{1}{8} b_1^2(1+4\ln 2-2\ln b_1) (\D_{\p\om} + \l)\Psi_k^{(0)}
\\
&\hphantom{+\tau \bigg(}-\frac{1}{2} \Big(2\ln 2-\ln b_1-\frac{3}{2} \Big) b_1 \nabla b_1\cdot\nabla \Psi_k^{(0)}
\\
&\hphantom{+\tau \bigg(}\pm \frac{1}{16} (3Kb_1^2+32 b_2^2+24b_1b_3)\Psi_k^{(1)}\pm C_4^{(k,1)}\bigg).
\end{aligned}
\label{4.56b}
\end{align}
Taking into account the obtained formulas and (\ref{4.2}), in accordance with the method of matching asymptotic expansions we conclude that
\begin{equation}\label{4.56c}
C_3^{(k,1)}(P,\e)=C_1^{(k,0)}(P,\e)\equiv0,
\end{equation}
while the solutions to the equation  (\ref{4.6b}) should satisfy the asymptotics
\begin{equation}
\phi_\pm^{(k)}(x',\e)=W_{2,1,\pm}^{(k)}(x',\e)+\frac{1}{\ln\e} W_{2,0,\pm}^{(k)}(x',\e) +o(\tau),\quad \tau\to0.\label{4.35b}
\end{equation}
Moreover, the identity
\begin{equation}\label{4.67}
\frac{1}{2}\left(\pm  \Psi_k^{(1)}  K -\D_{\p\om}\Psi_k^{(0)} -\l\Psi_k^{(0)}\right) =\Psi_k^{(2,\pm)}
\end{equation}
should hold.

\subsection{outer expansion: second term}

%D%

We substitute~(\ref{4.23}) and~(\ref{4.7}) into the eigenvalue equation for $\psi_\pm^{(k)}$ and equate
the coefficient of $\tau^0$. It  leads us to identity (\ref{4.67}).

We proceed to the problem (\ref{4.6b}), (\ref{4.35b}).
To study its solvability we shall make use of %\corr %D%two
one more auxiliary lemma. %D%s.
Recall that the matrices
$\mathrm{M}$ and $\widehat{\mathrm{M}}$ are defined in (\ref{3.8}) and (\ref{4.29a}), respectively.

\begin{lemma}\label{lm4.5}
The functions $f_{2,\pm}^{(k)}$ introduced in (\ref{4.6b})
satisfy the hypothesis of Lemma~\ref{lm4.3}. In particular, the asymptotics (\ref{4.78}) holds true with
\begin{equation}\label{4.90}
\begin{aligned}
&f_{-2}^\pm=\pm\frac{b_1^2}{8\ln\e}\Psi_k^{(1)},\quad f_{-3/2}^\pm=\frac{b_1 b_2}{4\ln\e} \Psi_k^{(1)},\quad
\\
&f_{-1}^\pm=-\frac{b_1^2}{4\ln\e}\left(\Psi_k^{(2,\pm)}-\frac{1}{b_1} \nabla b_1\cdot \nabla \Psi_k^{(0)} \mp K \Psi_k^{(1)}\right).
\end{aligned}
\end{equation}
\end{lemma}

\begin{proof}
We begin with an obvious identity
\begin{equation}\label{4.91}
f_{2,\pm}^{(k)}=\frac{1}{\ln\e}\left(-\Div_{x'} \mathrm{Q}_\pm \nabla_{x'}\psi_\pm^{(k)}+ \frac{1}{2} \big( \nabla_{x'} |\nabla_{x'}h_\pm|^2, \nabla_{x'} \psi_\pm^{(k)} \big)_{\mathds{R}^n}\right),
\end{equation}
which follows from the definition of $f_{2,\pm}^{(k)}$ in (\ref{4.6b}). To prove the lemma, we shall pass to the variables $(\tau,s)$ in the obtained identity. It follows from (\ref{4.8}), (\ref{4.55}) and the definition of $S_\e$ that
\begin{equation*}
h_\pm(x')=t,\quad \pm t>0.
\end{equation*}
Hence, by (\ref{4.9a}), (\ref{4.52})
\begin{equation}
h_\pm(x')=b(\pm\sqrt{\tau},P)= \sum\limits_{i=1}^{\infty} b_i(P)(\pm\sqrt{\tau})^i,\quad \tau\to+0.\label{4.92}
\end{equation}
Thus, employing (\ref{3.8}) and (\ref{4.91}), we conclude that the functions $f_{2,0,\pm}^{(k)}$ satisfy the hypothesis of Lemma~\ref{lm4.3} and in particular the asymptotics (\ref{4.78}) holds true. It remains to prove the identities (\ref{4.90}).

It follows from (\ref{4.44}) that
\begin{equation}\label{4.93}
|\nabla_{x'} h_\pm|^2 = \Big|\frac{\p h_\pm}{\p\tau}\Big|^2 + \nabla h_\pm\cdot (\mathrm{E}-\tau \mathrm{B}\mathrm{G}^{-1}_{\p\om})^{-2} \nabla h_\pm.
\end{equation}
We substitute (\ref{4.92}) into the obtained identity and arrive at the asymptotics for $|\nabla_{x'} h_\pm|^2$,
\begin{equation}\label{4.94}
\begin{aligned}
&|\nabla_{x'} h_\pm|^2=\sum\limits_{j=-2}^{\infty} h_{j/2}^\pm(P)\tau^{j/2},
\quad h_{-1}^\pm=\frac{1}{4}b_1^2, \quad h_{-1/2}^\pm =\pm b_1 b_2,\quad \tau\to+0.
\end{aligned}
\end{equation}
Employing these formulas and (\ref{3.8}), (\ref{4.29a}), (\ref{4.7}) and (\ref{4.44}) we rewrite the second term in the right hand side of (\ref{4.91}) as follows,
\begin{equation}\label{4.98}
\begin{aligned}
\frac{1}{2} \big( \nabla_{x'} |\nabla_{x'}h_\pm|^2,& \nabla_{x'} \psi_\pm^{(k)}\big)_\mathds{{R}^n}
=\frac{1}{2}  \frac{\p |\nabla_{x'} h_\pm|^2}{\p\tau} \frac{\p \psi_\pm^{(k)}}{\p\tau}
\\
&+\frac{1}{2} \nabla |\nabla_{x'} h_\pm|^2\cdot (\mathrm{E}-\tau \mathrm{B}\mathrm{G}^{-1}_{\p\om})^{-2} \nabla \psi_\pm^{(k)}
 =\sum\limits_{j=-4}^{\infty}  f_{j/2}^{\pm,2}\tau^{j/2},
\end{aligned}
\end{equation}
where $f_{j/2}^{\pm,2}\in C^\infty(\p\om)$ are some functions, and, in particular,
\begin{equation}\label{4.99}
\begin{aligned}
&f_{-2}^{\pm,2}=\mp \frac{1}{8\ln\e} b_1^2 \Psi_k^{(1)},\quad f_{-3/2}^{\pm,2}=- \frac{1}{4\ln\e} b_1 b_2 \Psi_k^{(1)},
\\
&f_{-1}^{\pm,2}=-\frac{b_1^2}{4\ln\e} \left( \Psi_k^{(2,\pm)} +\frac{1}{b_1} \nabla b_1\cdot \nabla \Psi_k^{(0)}\right).
\end{aligned}
\end{equation}
To obtain the same asymptotics for the first term in the right hand side of (\ref{4.91}), we employ first (\ref{4.43}),
\begin{equation}\label{4.100}
-\Div_{x'} \mathrm{Q}_\pm \nabla_{x'}\psi_\pm^{(k)}=- \frac{1}{\det \mathrm{M}} \Div_{(\tau,s)} (\det \mathrm{M}) \nabla_{(\tau,s)} h_\pm (\nabla_{(\tau,s)} h_\pm)^*\widehat{\mathrm{M}} \nabla_{(\tau,s)}\psi_\pm^{(k)}.
\end{equation}
It follows from the equations (\ref{4.23}), (\ref{4.29a}), (\ref{4.92}) that
\begin{align*}
(\nabla_{(\tau,s)}h_\pm)^* \widehat{\mathrm{M}} \nabla_{(\tau,s)}\psi_\pm^{(k)}=& \frac{\p h_\pm}{\p\tau} \frac{\p\psi_\pm^{(k)}}{\p\tau} + \nabla h_\pm\cdot (\mathrm{E}-\tau \mathrm{B}\mathrm{ G}^{-1}_{\p\om})^{-2} \nabla \psi_\pm^{(k)}
\\
=&\sum\limits_{j=-1}^{\infty} c_{j/2}^\pm \tau^{j/2},\quad \tau\to+0,
\\
(\det \mathrm{M}) \widehat{\mathrm{M}} \nabla_{(\tau,s)}h_\pm = & \sum\limits_{j=-1}^{\infty} \mathbf{c}_{j/2}^\pm \tau^{j/2},\quad \tau\to+0,
\end{align*}
where $c_{j/2}^\pm=c_{j/2}^\pm(P)\in C^\infty(\p\om)$ are some functions, $\mathbf{c}_{j/2}^\pm=\mathbf{c}_{j/2}^\pm(P)\in C^\infty(\p\om)$ are some $n$-dimensional vector-functions, and \begin{equation*}
c_{-1/2}^\pm=\frac{1}{2}b_1,\quad c_0^\pm= \pm b_2\Psi_k^{(1)},
\qquad
\mathbf{c}_{-1/2}^\pm=\pm\frac{1}{2}b_1 \mathbf{e}_1,\quad \mathbf{c}_0^\pm= b_2\mathbf{e}_1,
\end{equation*}
and $\mathbf{e}_1=(1,0,\ldots,0)^*$.  We substitute the last identities into (\ref{4.100}), which yields
\begin{align*}
&-\Div_{x'} \mathrm{Q}_\pm \nabla_{x'}\psi_\pm^{(k)}=\sum\limits_{j=-4}^{\infty} f_{j/2}^{\pm,1}\tau^{j/2},\quad \tau\to+0,
\\
&f_{-2}^{\pm,1}=\pm \frac{1}{4\ln\e} b_1^2\Psi_k^{(1)},\quad f_{-3/2}^{\pm,1}=\frac{1}{2\ln\e} b_1 b_2 \Psi_k^{(1)},\quad f_{-1}^{\pm,1}=\pm \frac{1}{4\ln\e} b_1^2 K \Psi_k^{(1)}.
\end{align*}
The last identity, (\ref{4.98}), (\ref{4.99}), (\ref{4.91}) imply the formulas (\ref{4.90}).
\end{proof}

Taking into account~(\ref{4.7}), we apply Lemma~\ref{lm4.5} to problem (\ref{4.6b}). It implies
that the right hand side of~(\ref{4.6b}) satisfies the hypothesis of Lemma~\ref{lm4.3} with the
first four coefficients given by~(\ref{4.90}).

Given some functions $V^{(0)}_k, V^{(1)}_k\in C^\infty(\p\om)$, suppose the solvability condition (\ref{4.79}) holds true. Then by (\ref{4.81}), (\ref{4.67}), (\ref{4.90}) there exists the unique solution to (\ref{4.6b}) with the asymptotics
\begin{equation}\label{4.95}
\begin{aligned}
\phi_\pm^{(k)}=&\frac{1}{\ln\e}\bigg(
\pm \frac{1}{8} b_1^2 \Psi_k^{(1)}\ln\tau  + b_1 b_2 \Psi_k^{(1)} \tau^{1/2}
\\
&+ \tau (1-\ln\tau) \left(-\frac{1}{4} b_1^2 \Psi_k^{(2,\pm)} + \frac{1}{4} b_1\nabla b_1\cdot \nabla \Psi_k^{(0)} \pm \frac{1}{8} K b_1^2 \Psi_k^{(1)}\right)\bigg)
\\
&+ U_k^{(0)} \pm V_k^{(0)} +\tau(V_k^{(1)}\pm U_k^{(1)})
+ \Odr(\tau^{3/2})
\\
=&\frac{1}{\ln\e}\bigg(
\pm \frac{1}{8} b_1^2 \Psi_k^{(1)}\ln\tau  + b_1 b_2 \Psi_k^{(1)} \tau^{1/2}
\\
&+ \tau (1-\ln\tau) \left(\D_{\p\om} + \frac{2}{b_1} \nabla b_1\cdot \nabla+ \l \right)\Psi_k^{(0)}\bigg)
\\
&+ U_k^{(0)} \pm V_k^{(0)} +\tau(V_k^{(1)}\pm U_k^{(1)}),\quad\tau\to+0,
\end{aligned}
\end{equation}
where $U^{(0)}_k, U^{(1)}_k\in C^\infty(\p\om)$ are some functions satisfying (\ref{4.81a}). We compare the last asymptotics with (\ref{4.56a}), (\ref{4.56b}), (\ref{4.35b}), take into consideration the identity (\ref{4.67}) and arrive at the formulas for $V_k^{(0)}$, $V_k^{(1)}$, $C_2^{(k,0}$ and $C_4^{(k,1)}$,
\begin{align*}
&V_k^{(0)}=-\frac{b_1^2}{4}\Psi_k^{(1)}+\frac{b_1^2}{8\ln\e} (1+4\ln 2-2\ln b_1)\Psi_k^{(1)},
\\
&C_2^{(k,0)}=\ln\e\, U_k^{(0)},
\\
&V_k^{(1)}=\frac{b_1^2}{4} \left(\D_{\p\om} + \frac{2}{b_1} \nabla b_1\cdot \nabla+ \l \right)\Psi_k^{(0)}
\\
&\hphantom{V_k^{(1)}=}-\frac{b_1^2}{4\ln\e}\Big( (2\ln 2-\ln b_1+1) (\D_{\p\om}+\l) \Psi_k^{(0)}
\\
&\hphantom{V_k^{(1)}=-\frac{1}{4\ln\e}\Big(}+\frac{4\ln 2-2\ln b_1-2}{b_1}\nabla b_1\cdot \nabla \Psi_k^{(0)}\Big),
\\
&C_4^{(k,1)}= \ln\e\, U_k^{(1)}-\frac{1}{16} (3 K b_1^2+32 b_2^2+24 b_1b_3)\Psi_k^{(1)}.
\end{align*}
In what follows the functions $V_k^{(0)}$, $V_k^{(1)}$, $C_2^{(k,0}$ and $C_4^{(k,1)}$ are supposed to be chosen in accordance with the above given formulas. Bearing these formulas, (\ref{4.67}) and (\ref{4.90}) in mind, we write the solvability conditions (\ref{4.79}) for the equation (\ref{4.6b}),
\begin{equation}\label{4.89}
\begin{aligned}
&\frac{1}{\ln\e}\lim\limits_{\d\to+0} \Bigg[  (f_{2,+}^{(k)},\psi_+^{(i)})_{L_2(\om^\d)}+ (f_{2,-}^{(k)},\psi_-^{(i)})_{L_2(\om^\d)}
 -\d^{-1/2} \int\limits_{\p\om} b_1 b_2 \Psi_k^{(1)} \Psi_i^{(0)}\di s
\\
&\hphantom{\ln\e\lim \Bigg[}+ \ln \d \int\limits_{\p\om} \frac{b_1^2}{4} \left(  \Psi_i^{(1)} \Psi_k^{(1)}+
 \Psi_i^{(0)} \bigg( \D_{\p\om} +\frac{2}{b_1}\nabla b_1\cdot \nabla+ \l\bigg)\Psi_k^{(0)}\right)
\di s\Bigg]
\\
&
 +\int\limits_{\p\om}\frac{b_1^2}{2 \ln\e}(2\ln 2-\ln b_1 +1) \Psi_i^{(0)}  (\D_{\p\om}+ \l)\Psi_k^{(0)}\di s
\\
&+ \int\limits_{\p\om} \frac{b_1}{\ln\e} (2\ln 2-\ln b_1-1) \Psi_i^{(0)} \nabla b_1\cdot\nabla \Psi_k^{(0)}\di s
\\
&+\int\limits_{\p\om}\frac{b_1^2}{2\ln\e}(2\ln 2- \ln b_1) \Psi_k^{(1)} \Psi_i^{(1)}\di s
\\
&-\int\limits_{\p\om}\frac{b_1^2}{2} \left(\Psi_k^{(1)} \Psi_i^{(1)}+  \Psi_i^{(0)} \bigg( \D_{\p\om} +\frac{2}{b_1}\nabla b_1\cdot \nabla+ \l\bigg)\Psi_k^{(0)}
\right)\di s
\\
&+ \mu_k\d_{ik}=0,\quad i,k=1,\ldots,m.
\end{aligned}
\end{equation}
Let us simplify the obtained identity. We first rewrite the formulas (\ref{4.6b}) of $f_{2,\pm}^{(k)}$ in a more convenient form  employing the eigenvalue equation for $\psi_\pm^{(k)}$ and the definition of the matrix $\mathrm{Q}_\pm$,
\begin{align*}
&f_{2,\pm}^{(k)}=-\Div_{x'} \Phi_\pm^{(k)}\nabla_{x'} h_\pm +\frac{\l}{2} |\nabla_{x'} h_\pm|^2 \psi_\pm^{(k)} + \frac{1}{2}\Div_{x'} |\nabla_{x'} h_\pm|^2\nabla_{x'}\psi_\pm^{(k)},
\\
&\Phi_\pm^{(k)}:=(\nabla_{x'}h_\pm,\nabla_{x'}\psi_\pm^{(k)})_{\mathds{R}^n}.
\end{align*}
Employing this representation, we integrate by parts to obtain
\begin{equation}\label{4.20}
\begin{aligned}
(f_{2,\pm}^{(k)},&\psi_\pm^{(i)})_{L_2(\om^\d)}= \int\limits_{\p\om^\d} \left( \Phi_\pm^{(k)}\frac{\p h_\pm}{\p\tau}-\frac{1}{2}|\nabla_{x'}h_\pm|^2\frac{\p\psi_\pm^{(i)}}{\p\tau}
\right)\psi_\pm^{(i)}\di s + \int\limits_{\om^\d} \Phi_\pm^{(i)} \Phi_\pm^{(k)}\di x'
\\
&+\frac{\l}{2}\int\limits_{\om^\d} |\nabla_{x'}h_\pm|^2 \psi_\pm^{(i)}\psi_\pm^{(k)}\di x'
- \frac{1}{2}\int\limits_{\om^d} |\nabla_{x'} h_\pm|^2 (\nabla_{x'}\psi_\pm^{(i)},\nabla_{x'}\psi_\pm^{(k)})_{\mathds{R}^d} \di x'.
\end{aligned}
\end{equation}
Applying (\ref{4.44}), we have
\begin{equation*}
\Phi_\pm^{(k)}=\frac{\p h_\pm}{\p\tau} \frac{\p\psi_\pm^{(k)}}{\p\tau} + \nabla h_\pm\cdot (\mathrm{E}-\tau \mathrm{B}\mathrm{G}^{-1}_{\p\om})^{-2} \nabla \psi_\pm^{(k)}
\end{equation*}
in a vicinity of $\p\om$. Hence, by~(\ref{4.7}),~(\ref{4.92}) and~(\ref{4.93}),
\begin{align}\label{4.25}
&\Phi_\pm^{(k)}=\frac{b_1}{2\sqrt{\tau}}\Psi_k^{(1)}+\Odr(1),\quad \tau\to+0,
\\
&\left( \Phi_\pm^{(k)}\frac{\p h_\pm}{\p\tau}-\frac{1}{2}|\nabla_{x'}h_\pm|^2\frac{\p\psi_\pm^{(i)}}{\p\tau}
\right)\psi_\pm^{(i)}\prod\limits_{j=1}^{n-1}(1-\tau K_j)\di s=\pm\frac{1}{8\tau} b_1^2 \Psi_i^{(1)} \Psi_k^{(1)}\nonumber
\\
&\hphantom{\Bigg(\Phi_\pm^{(k)}\frac{\p h_\pm}{\p\tau} }
+ \frac{1}{2\sqrt{\tau}} b_1 b_2 \Psi_i^{(0)}\Psi_k^{(1)}+ \frac{1}{8} b_1^2  \Psi_i^{(1)}\Psi_k^{(1)} \mp \frac{1}{8} b_1^2 K \Psi_i^{(0)}\Psi_k^{(1)}\nonumber
\\
&\hphantom{\Bigg(\Phi_\pm^{(k)}\frac{\p h_\pm}{\p\tau} }  + \frac{1}{4} (b_1^2 \Psi_k^{(2,\pm)} \pm 3 b_1 b_3 \Psi_k^{(1)} \pm 2 b_2^2 \Psi_k^{(1)} + 2 b_1\nabla b_1\cdot \Psi_k^{(0)})\Psi_i^{(0)} \nonumber
\\
&\hphantom{\Bigg(\Phi_\pm^{(k)} \frac{\p h_\pm}{\p\tau}}  + \Odr(\sqrt{\tau}),\quad \tau\to+0.\nonumber
\end{align}
Substituting the last identity into (\ref{4.20}) and using (\ref{4.21}) and (\ref{4.67}), we get
\begin{align*}
(f_{2,+}^{(k)},&\psi_+^{(i)})_{L_2(\om^\d)}+ (f_{2,-}^{(k)},\psi_-^{(i)})_{L_2(\om^\d)}
\\
&= \int\limits_{\om^\d} \frac{ |\nabla_{x'}h_+|^2}{2}\big(\l\psi_+^{(i)} \psi_+^{(k)} -(\nabla_{x'}\psi_+^{(i)}, \nabla_{x'} \psi_+^{(k)})_{\mathds{R}^d} \big)\di x'
\\
&+   \int\limits_{\om^\d} \frac{|\nabla_{x'}h_-|^2}{2} \big(\l\psi_-^{(i)} \psi_-^{(k)} -(\nabla_{x'}\psi_-^{(i)}, \nabla_{x'} \psi_-^{(k)})_{\mathds{R}^d} \big)\di x'
\\
&+\int\limits_{\om^\d} (\Phi_+^{(i)}\Phi_+^{(k)}+\Phi_-^{(i)}\Phi_-^{(k)}) \di x' + \d^{-1/2} \int\limits_{\p\om} b_1 b_2 \Psi_i^{(0)} \Psi_k^{(0)}\di s
\\
&+\int\limits_{\p\om} \frac{b_1^2}{4} \Psi_i^{(1)} \Psi_k^{(1)} \di s - \int\limits_{\p\om} \frac{b_1^2}{4} \Psi_i^{(0)}(\D_{\p\om}+\l) \Psi_k^{(0)}\di s
\\
&+ \int\limits_{\p\om} b_1 \Psi_i^{(0)}\nabla b_1 \cdot \nabla \Psi_k^{(0)}\di s + \Odr(\d^{1/2}),\quad \d\to+0.
\end{align*}
We integrate by parts once again, this time over $\p\om$, we %D%get
have % \corr
\begin{equation}\label{4.103}
\int\limits_{\p\om} b_1^2 \Psi_i^{(0)}  \bigg( \D_{\p\om} +\frac{2}{b_1}\nabla b_1\cdot \nabla+ \l\bigg)\Psi_k^{(0)}
\di s=\int\limits_{\p\om} b_1^2 \big(\l\Psi_i^{(0)}\Psi_k^{(0)}-\nabla\Psi_i^{(0)}\cdot \nabla\Psi_k^{(0)}   \big)
\di s.
\end{equation}
Substituting two the last identities %\corr %D% and (\ref{4.24}) %
into (\ref{4.89}) yields
\begin{equation}\label{4.51}
\begin{aligned}
&\frac{1}{\ln\e}\lim\limits_{\d\to+0} \Bigg[
\int\limits_{\om^\d} \frac{|\nabla_{x'}h_+|^2}{2}\big(\l\psi_+^{(i)} \psi_+^{(k)} -(\nabla_{x'}\psi_+^{(i)}, \nabla_{x'} \psi_+^{(k)})_{\mathds{R}^d} \big)\di x'
\\
&\hphantom{\lim\limits_{\d\to+0} \Bigg[}+ \int\limits_{\om^\d} \frac{|\nabla_{x'}h_-|^2}{2}\big(\l\psi_-^{(i)} \psi_-^{(k)} -(\nabla_{x'}\psi_-^{(i)}, \nabla_{x'} \psi_-^{(k)})_{\mathds{R}^d} \big)\di x'
\\
&\hphantom{\lim\limits_{\d\to+0} \Bigg[}+\int\limits_{\om^\d} (\Phi_+^{(i)}\Phi_+^{(k)}+\Phi_-^{(i)}\Phi_-^{(k)}) \di x'
\\
&\hphantom{\lim\limits_{\d\to+0} \Bigg[}+ \ln \d \int\limits_{\p\om} \frac{b_1^2}{4} \big( \Psi_i^{(1)} \Psi_k^{(1)}+\l\Psi_i^{(0)}\Psi_k^{(0)} -\nabla\Psi_i^{(0)}\cdot \nabla\Psi_k^{(0)}\big)\di s\Bigg]
\\
&+\int\limits_{\p\om}\frac{b_1^2}{4\ln\e}(1+4\ln 2-2\ln b_1)
\big(\Psi_i^{(1)}\Psi_k^{(1)}+
\Psi_i^{(0)}  (\D_{\p\om}+ \l)\Psi_k^{(0)}
\big)\di s
\\
&+\int\limits_{\p\om} \frac{b_1}{\ln\e} (2\ln 2-\ln b_1) \Psi_i^{(0)} \nabla b_1\cdot\nabla \Psi_k^{(0)}\di s
\\
&-\int\limits_{\p\om}\frac{b_1^2}{2} \left(\Psi_k^{(1)} \Psi_i^{(1)}+  \Psi_i^{(0)} \bigg( \D_{\p\om} +\frac{2}{b_1}\nabla b_1\cdot \nabla+ \l\bigg)\Psi_k^{(0)}
\right)\di s
\\
&+ \mu_k\d_{ik}=0,
\end{aligned}
\end{equation}
as $i,k=1,\ldots,m$. It follows from (\ref{4.25}), (\ref{4.94}) and~(\ref{4.7}) that
\begin{align*}
& |\nabla_{x'}h_+|^2\big(\l\psi_+^{(i)} \psi_+^{(k)} -(\nabla_{x'}\psi_+^{(i)}, \nabla_{x'} \psi_+^{(k)})_{\mathds{R}^d}\big)
\\
&+ |\nabla_{x'}h_-|^2\big(\l\psi_-^{(i)} \psi_-^{(k)} -(\nabla_{x'}\psi_-^{(i)}, \nabla_{x'} \psi_-^{(k)})_{\mathds{R}^d} \big)
\\
&\hphantom{\Phi_\pm^{(i)}\Phi_\pm^{(k)}}=\frac{b_1^2}{2\tau} (\l\Psi_i^{(0)}\Psi_k^{(0)}-\nabla\Psi_i^{(0)}\cdot \nabla\Psi_k^{(0)})+\Odr(\tau^{-1/2}),\quad \tau\to+0,
\\
&\Phi_\pm^{(i)}\Phi_\pm^{(k)}=\frac{b_1^2}{4\tau} \Psi_i^{(1)}\Psi_k^{(1)} + \Odr(\tau^{-1/2}),\quad \tau\to+0.
\end{align*}
Hence, the limit in (\ref{4.51}) is finite. To calculate the boundary integrals in (\ref{4.51}) we integrate by
parts as follows
\begin{align*}
&\int\limits_{\p\om}\frac{b_1^2}{4}(1+4\ln 2-2\ln b_1)
\big(\Psi_i^{(1)}\Psi_k^{(1)}+
\Psi_i^{(0)}  (\D_{\p\om}+ \l)\Psi_k^{(0)}
\big)\di s
\\
&+ \int\limits_{\p\om} b_1 (2\ln 2-\ln b_1) \Psi_i^{(0)} \nabla b_1\cdot\nabla \Psi_k^{(0)}\di s
\\
&=\int\limits_{\p\om}\frac{b_1^2}{4}(1+4\ln 2-2\ln b_1)
\big(\Psi_i^{(1)}\Psi_k^{(1)}+\l\Psi_i^{(0)}  \Psi_k^{(0)}-\nabla\Psi_i^{(0)}\cdot \nabla\Psi_k^{(0)}
\big)\di s.
\end{align*}
Due to this identity, (\ref{4.103}), the definition of $b_1$ in (\ref{4.52}) and the definitions (\ref{2.11a})
and (\ref{2.11b}) of the matrices $\Lambda^{(0)}$ and $\Lambda^{(1)}$, respectively,  we can rewrite (\ref{4.51}) in the final form
\begin{equation*}
\mu_k\d_{ik}=\L^{(0)}_{ik}+\frac{1}{\ln\e}\L^{(1)}_{ik}.
\end{equation*}
Since the matrix on the right hand side of the last identity is diagonal, we conclude that
the solvability condition for the problem (\ref{4.6b}), (\ref{4.35b}) is satisfied provided
$\mu_k$ are the eigenvalues of the matrix $\Lambda^{(0)}+\frac{1}{\ln\e}\Lambda^{(1)}$. It
follows from \cite[Ch. I\!I, Sec. 6.1, Th. 6.1]{K} that the eigenvalues of this matrix are
holomorphic in $\frac{1}{\ln\e}$ and converge to those of $\Lambda^{(0)}$ as $\e\to0$.

In view of the choice of $\mu_i$ the problems (\ref{4.6b}), (\ref{4.95}) are solvable. We observe that each of the functions $\phi_\pm^{(k)}$ is
defined up to a linear combination of the eigenfunctions $\psi_\pm^{(i)}$. The exact values of the coefficients of these linear combinations can be
determined
while constructing the next terms in the asymptotic expansions for $\l_k(\e)$ and $\psi_\e^{(k)}$. The formal constructing of the asymptotic
expansions is complete. %\corr %D%We proceed to the %D%justification of these asymptotics.

\subsection{Justification of the asymptotics}

In order to justify the obtained asymptotics, one has to construct additional terms. This is a general and standard situation for singularly perturbed
problems. In our case one should construct the terms of the order up to $\Odr(\e^4)$ in the outer expansion for the eigenfunctions and for the
eigenvalues, and the terms of order up to $\Odr(\e^6)$ in the inner expansion for the eigenfunctions. The asymptotics with the additional terms
read as follows, %\corr\corr
\begin{equation}
\begin{aligned}
&\l_k(\e)=\l+\e^2\ln\e\,\mu_k%D%
\left(\frac{1}{\ln\e}\right)
+\e^4\ln^2\e\,\eta_{k}(\e)+\ldots,
\\
&\psi_{\e,ex}^{(k)}=\PS_\e(\bs{\psi}_k+\e^2\ln\e\,\bs{\phi}_k
+\e^4\ln^2\e\,\bs{\theta}_{k}+\ldots),
\\
&\psi_{\e,in}^{(k)}=v_0^{(k)}+\sum\limits_{i=2}^{6}\e^i v_i^{(k)}+\ldots,
\end{aligned}\label{4.116}
\end{equation}
where $\bs{\theta}_{k}=(\theta_{+}^{(k)},\theta_{-}^{(k)})$, $\theta_{\pm}^{(k)}=\theta_{\pm}^{(k)}(x',\e)$, $v_i^{(k)}=v_i^{(k)}(\xi,P,\e)$, and we used that $v_1^{(k)}=0$ by (\ref{4.46b}), (\ref{4.56c}).
The equations for $\theta_{\pm}^{(k)}$ are
\begin{align*}
&(-\D_{x'}-\l)\theta_{\pm}^{(k)}=\frac{1}{\ln\e} \mathcal{H}_\pm^{(2)}\phi_\pm^{(k)} +\frac{1}{\ln^2\e} \mathcal{H}_\pm^{(4)}\psi_\pm^{(k)} +
\mu_k\phi_\pm^{(k)}+\eta_{k}\psi_\pm^{(k)},\quad x'\in\om_\pm,
\\
&
\mathcal{H}_\pm^{(4)}:=\frac{3}{8} |\nabla_{x'} h_\pm|^4\D_{x'}-\frac{1}{2}|\nabla_{x'} h_\pm|^2\Div_{x'} \left(\frac{1}{2}|\nabla_{x'}h_\pm|^2\mathrm{E}- \mathrm{Q}_\pm\right)\nabla_{x'}
\\
&\hphantom{\mathcal{H}_\pm^{(4)}:=}-\Div_{x'} \left( \frac{1}{8}|\nabla_{x'} h_\pm|^4 \mathrm{E}+\frac{1}{2} \mathrm{Q}_\pm |\nabla_{x'}h_\pm|^2 +\mathrm{Q}_\pm^2\right)\nabla_{x'}.
\end{align*}
The functions $\theta_{\pm}^{(k)}$ should satisfy the asymptotics
\begin{align*}
&\theta_{\pm}^{(k)}(x',\e)=W_{4,2,\pm}^{(k)}(x',\e)+\frac{1}{\ln\e}
W_{4,1,\pm}^{(k)}(x',\e)+\frac{1}{\ln^2\e}
W_{4,0,\pm}^{(k)}(x',\e)+o(1),\quad \tau\to+0,
\\
&W_{4,2,\pm}^{(k)}=-\frac{1}{32}b_1^3\left( b_1(\D_{\p\om}+\l)\Psi_k^{(0)}+2\nabla b_1\cdot \nabla \Psi_k^{(0)}\right),
\\
&W_{4,1,\pm}^{(k)}=\frac{1}{32}b_1^3\left(\ln\tau+1+ 4 \ln 2- 2\ln b_1\right)\big(b_1 (\D_{\p\om}+\l)\Psi_k^{(0)} + 2
\nabla b_1\cdot \nabla \Psi_k^{(0)}\big),
\\
&W_{4,0,\pm}^{(k)}=\pm \frac{1}{128}\frac{\Psi_k^{(1)}b_1^4}{\tau}+\frac{1}{8} \frac{\Psi_k^{(1)} b_1^3 b_2}{\sqrt{\tau}}
\\
&\hphantom{W_{4,0,\pm}^{(k)}=}- \frac{1}{128} b_1^3 \big(b_1 (\D_{\p\om}+\l)\Psi_k^{(0)} + 2
\nabla b_1\cdot \nabla \Psi_k^{(0)}\big)(\ln\tau+4\ln 2-2\ln b_1+1)^2
\\
&\hphantom{W_{4,0,\pm}^{(k)}=}-\frac{1}{128} b_1^3  \big(b_1 (\D_{\p\om}+\l)\Psi_k^{(0)} - 2
\nabla b_1\cdot \nabla \Psi_k^{(0)}\big)
\\
&\hphantom{W_{4,0,\pm}^{(k)}=} \pm \frac{1}{256}\Psi_k^{(1)} \left(3 K b_1^4+48 b_1^3 b_3+128 b_1^2 b_2^2\right).
\end{align*}
The equations for the functions $v_5^{(k)}$, $v_6^{(k)}$ are obtained in the same way as those for $v_i^{(k)}$, $i=0,\ldots,4$, from
\begin{align*}
&\mathcal{L}_{-4} v_5^{(k)}+ \sum\limits_{i=-3}^{-1}  \mathcal{L}_i v_{1-i}^{(k)}
\mathcal{L}_1 v_0^{(k)}=0 \quad \text{on}\quad
\mathds{R}\times\p\om,
\\
&\mathcal{L}_{-4} v_6^{(k)}+ \sum\limits_{i=-3}^{0}  \mathcal{L}_i v_{2-i}^{(k)} + \mathcal{L}_2 v_0^{(k)}
=\l v_2^{(k)}+\ln\e\,\eta_{k} v_0^{(k)} \quad
\text{on}\quad \mathds{R}\times\p\om,
\end{align*}
where the operators $\mathcal{L}_1$, $\mathcal{L}_2$ are the next terms in the expansion (\ref{4.32}). It can be shown that the problem
for $\theta_{\pm}^{(k)}$ is solvable for some $\eta_{k}(\e)$. The equations for $v_5^{(k)}$
and $v_6^{(k)}$ can be solved explicitly. The arbitrary coefficients $C_{5,1}^{(k)}$, $C_{5,0}^{(k)}$, $C_{6,1}^{(k)}$, $C_{6,0}^{(k)}$
appearing in $v_5^{(k)}$, $v_6^{(k)}$ can be determined while matching
the inner and outer expansions.

We now introduce the partial sums %\corr\corr
\begin{align*}
&\widehat{\l}_\e^{(k)}=\l+\e^2\ln\e\,\mu_k%D%
\left(\frac{1}{\ln\e}\right)+\e^4\ln^2\e\,\eta_{k}(\e),
\\
&\widehat{\psi}_{\e,ex}^{(k)}=\PS_\e(\bs{\psi}_k+\e^2\ln\e\,\bs{\phi}_k
+\e^4\ln^2\e\,\bs{\theta}_{k}),
\\
&\widehat{\psi}_{\e,in}^{(k)}=v_0^{(k)}+\sum\limits_{i=2}^{6}\e^i v_i^{(k)}
\end{align*}
and define the final approximation for the eigenfunctions as
\begin{equation*}
\widehat{\psi}_\e^{(k)}(x)=\widehat{\psi}_{\e,ex}^{(k)}(x)\chi \left(\frac{\tau}{\e^\a}\right)+\widehat{\psi}_{\e,in}^{(k)}(\xi,P) \left(1-\chi \left(\frac{\tau}{\e^\a}\right)\right),
\end{equation*}
where $a\in(0,1)$ is a fixed constant, and $\chi$ is the cut-off function introduced in the proof of Lemma~\ref{lm3.3}.

\begin{lemma}\label{lm4.6}
The function $\widehat{\psi}_\e^{(k)}\in C^\infty(S_\e)$
satisfies the convergence
\begin{equation}\label{4.106}
\|\widehat{\psi}_\e^{(k)}-\PS_\e\bs{\psi}_k\|_{L_2(S_\e)}\to0,\quad \e\to+0,
\end{equation}
and the equation
\begin{equation}\label{4.107}
(\mathcal{H}_\e-\widehat{\l}_\e^{(k)})\widehat{\psi}_\e^{(k)}=F_\e^{(k)},
\end{equation}
where for the right hand side the uniform in $\e$ estimate
\begin{equation}\label{4.108}
\|F_\e^{(k)}\|_{L_2(S_\e)}\leqslant C\e^{5\a/2}
\end{equation}
holds true. The relations
\begin{equation}\label{4.109}
(\PS_\e\bs{\psi}_i,\PS_\e\bs{\psi}_j)_{L_2(S_\e)}\to\d_{ij},\quad \e\to+0,
\end{equation}
are valid.
\end{lemma}

The proof of this lemma is not very difficult and is based on lengthy and rather technical,
but straightforward calculations. Because of this, and in order not to overload the text with
long technical formulas we shall skip these here.

It follows from Lemma~\ref{lm3.4} and equation (\ref{4.107}) that
\begin{equation}\label{4.110}
\widehat{\psi}_\e^{(k)}=
\sum\limits_{i=1}^{m}
\frac{\psi_\e^{(i)}}{\l_i(\e)-\l_k(\e)} (F_\e^{(k)},
\psi_\e^{(i)})_{L_2(S_\e)} + \mathcal{R}_\e(\l_k(\e))F_\e^{(k)},
\end{equation}
and, by~(\ref{4.108}),
\begin{equation}\label{4.111}
\|\mathcal{R}_\e(\l_k(\e))F_\e^{(k)}\|_{\H^1(S_\e)}\leqslant C\e^{5\a/2},\quad k=1,\ldots,m,
\end{equation}
where the constant $C$ is independent of $\e$. We calculate the scalar products of the functions $\widehat{\psi}_\e^{(k)}$ in $L_2(S_\e)$ taking into consideration (\ref{4.110}) and the properties of the operator $\mathcal{R}_\e$ described in Lemma~\ref{lm3.4}:
\begin{align*}
&(\widehat{\psi}_\e^{(k)},\widehat{\psi}_\e^{(p)})_{L_2(S_\e)} = \sum\limits_{i=1}^{m} \g_i^{(k)}(\e) \g_i^{(p)}(\e) +
(\mathcal{R}_\e(\l_k(\e))F_\e^{(k)}, \mathcal{R}_\e(\l_\e^{(p)})F_\e^{(p)})_{L_2(S_\e)},
\\
&\g_\e^{(k)}(\e):=\frac{1}{\l_i(\e)-\widehat{\l}_\e^{(k)}} (F_\e^{(k)},
\psi_\e^{(i)})_{L_2(S_\e)}.
\end{align*}
The identities obtained and (\ref{4.111}), (\ref{4.106}), (\ref{4.109}) yield
\begin{equation}\label{4.115}
\sum\limits_{i=1}^{m} \g_i^{(k)}(\e) \g_i^{(p)}(\e)\to \d_{kp},\quad \e\to+0.
\end{equation}
In particular, as $p=k$  it implies
\begin{equation}\label{4.113}
|\g_i^{(k)}(\e)|\leqslant \frac{3}{2}
\end{equation}
for sufficiently small $\e$. We introduce the matrix $\mathrm{R}_\e:=\big(\g_i^{(k)}(\e)\big)$ and rewrite (\ref{4.115}) as $\mathrm{R}_\e\mathrm{R}_\e^*\to \mathrm{E}$, $\e\to+0$, where $^*$ denotes matrix transposition.
Thus, $|\Det \mathrm{R}_\e|\to1$ as $\e\to+0$. Therefore, for each sufficiently small $\e$ there exists a
permutation $\big(i_1(\e),i_2(\e),\ldots,i_m(\e)\big)$ such that
\begin{equation}\label{4.114}
\left|\prod\limits_{i=1}^{m} \g_{i_k(\e)}^{(k)}(\e)\right|\geqslant \frac{1}{2m!}.
\end{equation}
For a given $\e$ we rearrange the eigenvalues $\l_i(\e)$ and $\widehat{\psi}_\e^{(k)}$ so that $i_k(\e)=k$ that by (\ref{4.113}), (\ref{4.114}) it yields
\begin{equation*}
|\g_i^{(i)}(\e)|\geqslant \frac{2^{m-2}}{3^{m-1}m!}, \quad i=1,\ldots,m.
\end{equation*}
In view of the definition of $\g_k^{(k)}(\e)$, (\ref{4.108}), and the normalization of $\psi_\e^{(i)}$ it follows
\begin{equation*}
|\l_i(\e)-\l_i(\e)|\leqslant \frac{3^{m-1}m!}{2^{m-2}} \big|(F_\e^{(i)},\psi_\e^{(i)})_{L_2(S_\e)}\big|\leqslant C\e^{5\a/2}.
\end{equation*}
Choosing $\a>4/5$, we arrive at the asymptotics (\ref{2.16}).

Denote now
\begin{equation*}
\widetilde{\psi}_\e^{(k)}=\PS_\e(\bs{\psi}_k+\e^2\ln\e\,\bs{\phi}_k)\chi \left(\frac{\tau}{\e^\a}\right)+\left(v_0^{(k)}+\sum\limits_{i=2}^4\e^i v_i^{(k)}\right)\left(1-\chi \left(\frac{\tau}{\e^\a}\right)\right).
\end{equation*}
By direct calculations one can check that
\begin{equation*}
\|\widehat{\psi}_\e^{(k)}-\widetilde{\psi}_\e^{(k)}\|_{\H^1(S_\e)}=\Odr(\e^{\frac{5\a}{2}}).
\end{equation*}
This identity and (\ref{4.111}) imply
\begin{equation*}
\sum\limits_{i=1}^{m}
\g_i^{(k)}(\e)\psi_\e^{(i)}=\psi_\e^{(k)}+\Odr(\e^{\frac{5\a}{2}}),\quad k=1,\ldots,m.
\end{equation*}
Since the right hand sides of these identities are linear independent, the functions $\sum\limits_{i=1}^{m}
\g_i^{(k)}(\e)\psi_\e^{(i)}$ form a basis spanned over the eigenfunctions $\psi_\e^{(i)}$, $i=1,\ldots,m$. Hence, we arrive at
\begin{theorem}\label{th4.7}
Let $\mathcal{P}_\e$ be the total projector associated with the eigenvalues $\l_i(\e)$, $i=1,\ldots,m$, $\widetilde{\mathcal{P}}_\e$ be the projector on the space spanned over $\widetilde{\psi}_\e^{(i)}$, $i=1,\ldots,m$. Then
\begin{equation*}
\mathcal{P}_\e=\widetilde{\mathcal{P}}_\e+\Odr(\e^{2+\rho}),
\end{equation*}
where $\rho$ is any constant in $(0,1/2)$.
\end{theorem}

\section*{Acknowledgments}{Both authors were partially supported by FCT's projects PTDC/\ MAT/\ 101007/2008
and PEst-OE/MAT/UI0208/2011. D.B. was partially supported by RFBR  and by Federal Task Program (contract 02.740.11.0612).
}


\begin{thebibliography}{99999}

\bibitem{abfr}
M. Abreu and P. Freitas,
On the invariant spectrum of $S^{1}-$invariant metrics
on $S^{2}$, Proc. London Math. Soc. {\bf 84} (2002), 213--230

\bibitem{bc} D. Borisov, and G. Cardone. Complete asymptotic expansions for the eigenvalues of the Dirichlet Laplacian in
thin three-dimensional rods, ESAIM:COCV {\bf17} (2011), 887-908.

\bibitem{bofr1}
D. Borisov and P. Freitas,
Singular asymptotic expansions for Dirichlet eigenvalues and eigenfunctions of the
Laplacian on thin planar domains,
Ann. Inst. H. Poincar\'{e} Anal. Non Lin\'{e}aire {\bf 26} (2009) 547-560.

\bibitem{bofr2}
D. Borisov and P. Freitas,
Asymptotics of Dirichlet eigenvalues and eigenfunctions of the
Laplacian on thin domains in $\mathbb{R}^d$,
J. Funct. Anal. {\bf 258} (2010), 893--912, doi:10.1016/j.jfa.2009.07.014.

\bibitem{bofr3}
D. Borisov and P. Freitas,
Eigenvalue asymptotics for almost flat compact hypersurfaces,
Dokl. Akad. Nauk. 442 (2012), 151--155; translation in Dokl. Math. 85 (2012), 18-22.

\bibitem{cde}
B. Colbois, E. Dryden and A. El Soufi,
Extremal G-invariant eigenvalues of the Laplacian of G-invariant metrics,
Math. Zeit. {\bf 258} (2008), 29--41.

\bibitem{EP3} P. Exner and O. Post, Convergence of spectra of graph-like thin manifolds, J. Geom. Phys.,  \textbf{54} (2005),  77-115.

\bibitem{EP1} P. Exner, O. Post, Approximation of quantum graph vertex couplings by scaled Schr\"{o}dinger operators
on thin branched manifolds, J. Phys. A, \textbf{42} (2009), id 415305.

\bibitem{frei}
P. Freitas,
Precise bounds and asymptotics for the first Dirichlet eigenvalue of
triangles and rhombi, J. Funct. Anal. {\bf 251} (2007), 376--398.

\bibitem{frso}
L. Friedlander and M. Solomyak, On the spectrum of the Dirichlet
Laplacian in a narrow strip, Israel J. Math. {\bf 170} (2009), 337--354.

\bibitem{Gr2} D. Grieser, Thin tubes in mathematical physics, global analysis and spectral geometry,
Proceedings of Symposia in Pure Mathematics ``Analysis on Graphs and Its Applications'',  \textbf{77},
(2008), 565-593.

\bibitem{Il} A.M. Il'in, \emph{Matching of asymptotic expansions of solutions of boundary value problems}.
Translations of Mathematical Monographs. 102. Providence, RI: American Mathematical Society, 1992.

\bibitem{K} T. Kato. Perturbation theory for linear operators.
    Springer-Verlag, Berlin, Heidelberg, New York, 1966.

\bibitem{krzu}
D.~Krej\v{c}i\v{r}\'{\i}k and E. Zuazua, The Hardy inequality and the heat equation in twisted tubes,
J. Math. Pures Appl. {\bf 94} (2010), 277--303.

\bibitem{Na} S.A. Nazarov, \emph{Dimension reduction and integral estimates, Asymptotic theory of
thin plates and rods 1.} Novosibirsk, Nauchnaya Kniga, 2001.

\bibitem{Pa} G. Panasenko, \emph{Multi-scale modelling for structures and composites}. Springer, 2005.

\bibitem{RS1} M. Reed and B. Simon. \emph{Methods of mathematical physics. Functional analysis}, Academic Press, 1980.


\end{thebibliography}
\end{document}